\documentclass[preprint, 11pt]{elsarticle}

\usepackage[margin=2.5cm]{geometry}

\usepackage{amssymb}
\usepackage{amsthm}

\usepackage{algorithm}
\usepackage{algpseudocode}
\usepackage{mathrsfs}
\usepackage{xcolor}
\usepackage{amsmath}
\usepackage{mathtools}
\mathtoolsset{showonlyrefs}
\usepackage{subcaption}
\usepackage{siunitx}
\usepackage{enumitem}
\newlist{steps}{enumerate}{1}
\setlist[steps, 1]{label = \underline{Step \arabic*}.}

\newtheorem{theorem}{Theorem}
\newtheorem{remark}{Remark}

\newtheorem{lemma}{Lemma}

\newtheorem{assumption}{Assumption}

\def\R{\mathbb{R}}
\journal{TBA}

\def\figsize{0.48}

\usepackage[colorlinks=true,linkcolor=blue,citecolor=blue,urlcolor=blue]{hyperref}

\begin{document}
\begin{frontmatter}



\title{Generalized convergence of the deep BSDE method: a step towards fully-coupled FBSDEs and applications in stochastic control}

\author[label1]{Balint Negyesi\corref{cor1}}
\ead{b.negyesi@uu.nl}
\author[label2]{Zhipeng Huang}
\author[label2]{Cornelis W. Oosterlee}
\affiliation[label1]
            {organization={Delft Institute of Applied Mathematics (DIAM), Delft University of Technology}, addressline={P.O. Box 5031},
postcode={2600 GA},
city={Delft},
country={The Netherlands}}
\affiliation[label2]{organization={Mathematical Institute, Utrecht University}, addressline={Postbus 80010}, postcode={3508 TA}, city={Utrecht}, country={The Netherlands}}
\cortext[cor1]{Corresponding author}
\begin{abstract}
We are concerned with high-dimensional coupled FBSDE systems approximated by the deep BSDE method of Han et al. (2018). 
It was shown by Han and Long  (2020) that the errors induced by the deep BSDE method admit a posteriori estimate depending on the loss function, whenever the backward equation only couples into the forward diffusion through the $Y$ process. We generalize this result to drift coefficients that may also depend on $Z$, and give sufficient conditions for convergence under standard assumptions. The resulting conditions are directly verifiable for any equation. Consequently, unlike in earlier theory, our convergence analysis enables the treatment of FBSDEs stemming from stochastic optimal control problems. In particular, we provide a theoretical justification for the non-convergence of the deep BSDE method observed in recent literature, and present direct guidelines for when convergence can be guaranteed in practice. Our theoretical findings are supported by several numerical experiments in high-dimensional settings.
\end{abstract}


\end{frontmatter}

\section{Introduction}\label{sec1}

In this paper, we are concerned with the numerical approximation of a system of coupled forward-backward stochastic differential equations (FBSDE) over a finite time interval $[0, T]$
\begin{equation}\label{eq:FBSDE}
\left\{
    \begin{aligned}
        X_t &= x_0 + \int_0^t b(s, X_s, Y_s, Z_s) \mathrm{d}s + \int_0^t \sigma(s, X_s, Y_s) \mathrm{d}W_s,\\
        Y_t &= g(X_T) + \int_t^T f(s, X_s, Y_s, Z_s) \mathrm{d}s - \int_t^T Z_s\mathrm{d}W_s,
    \end{aligned}
\right.
\end{equation}
where $b: [0,T]\times \mathbb{R}^{d} \times \mathbb{R}^{q} \times \mathbb{R}^{q\times m} \to \mathbb{R}^{d} $, $\sigma: [0,T]\times \mathbb{R}^{d} \times \mathbb{R}^{q} \to \mathbb{R}^{d\times m} $, $f:[0,T]\times \mathbb{R}^{d} \times \mathbb{R}^{q} \times \mathbb{R}^{q\times m} \to \mathbb{R}^{q}$ and $g: \mathbb{R}^{d} \to \mathbb{R}^{q} $ are all deterministic mappings. 
The equation is given on a complete probability space $(\Omega, \mathcal{F}, \mathbb{P})$, over which $\{W_t\}_{0\leq T}$ is a standard $m$-dimensional Brownian motion, $\mathcal{F}:=\{\mathcal{F}_t\}_{0 \leq t \leq T}$ its natural filtration and augmented by the usual $\mathbb{P}$-null sets. A triple of $(\mathbb{R}^d\times \mathbb{R}^q\times \mathbb{R}^{q\times m})$ valued, $\mathcal{F}_t$ adapted stochastic processes $\{(X_t, Y_t, Z_t)\}_{0\leq t \leq T}$
is a solution if \eqref{eq:FBSDE} holds $\mathbb{P}$ almost surely and the processes satisfy natural integrability conditions, see \cite{zhang2017backward, ma_forward-backward_2007}.

The study of FBSDEs was initiated by the seminal paper of Pardoux and Peng \cite{pardouxpeng}, and then extended to coupled equations by Antonelli \cite{antonelli_backward-forward_1993}. Equations like \eqref{eq:FBSDE} subsequently attracted widespread attention due to their inherent connections with systems of second-order quasi-linear partial differential equations (PDE) established by non-linear extension to the Feynman-Kac lemma, see e.g. \cite{ma_forward-backward_2007, zhang2017backward} and theorem \ref{uniquesol} below. This probabilistic representation casts FBSDEs to be the natural framework to model a wide range of problems arising in mathematical finance, physics, biology and stochastic control. The well-posedness of \eqref{eq:FBSDE} has been rigorously studied and established under by now standard assumptions, see e.g. \cite{ma_forward-backward_2007, pardoux1999forward, delarue_existence_2002, delarue_weak_2006} and the references therein. Most of such results rely either on classical solutions of the corresponding quasi-linear PDE with high regularity or some abstract conditions heuristically associated with monotonicity, small time duration or weak coupling. In the rest of the article, we consider the setting where \eqref{eq:FBSDE} is well-posed and admits a unique strong solution triple.

Solving FBSDEs analytically is seldom possible and one usually has to resort to numerical approximations. In the decoupled framework where $b, \sigma$ in \eqref{eq:FBSDE} do not depend on $Y, Z$, one can detach the solution of $X$, and use the resulting discrete time approximation in order to approximate the solution pair of the BSDE by sequences of backward, recursive conditional expectations, we refer to \cite{gobet_regression-based_2005, bouchard_discrete-time_2004, bender2007forward, bender_least-squares_2012, ruijter_fourier_2015}. In the coupled framework, things become more subtle due to the interdependence from the backward equation into the forward diffusion. Inspired by \cite{ma_solving_1994}, classical approaches usually consist of \emph{decoupling} the forward diffusion by means of deterministic mappings, and iteratively converging to the unique decoupling field related to the associated quasi-linear PDE, see e.g. \cite{ma_solving_1994, douglas_numerical_1996, cvitanic_steepest_2005, delarue_forward_2006, bender2008time, huijskens_efficient_2016}. A common challenge across the aforementioned classical references is the setting of high-dimensionality. In fact, whenever either $d, q$ or $m$ are large, these methods suffer from the curse dimensionality and become intractable in high dimensions. Such settings arise naturally, for instance, in portfolio allocation or climate risk management.
A recently emerging branch of numerical algorithms called \emph{deep BSDE} methods and pioneered by \cite{han2018, e_deep_2017} addressed this gap, and has shown remarkable empirical performance in terms of tackling high-dimensional FBSDEs and associated quasi-linear PDEs. These approaches were first developed for decoupled equations, see e.g. \cite{han2018, e_deep_2017, hure2020deep, gao_convergence_2023, negyesi_one_2024}, and then extended to the coupled framework \cite{hanlong2020, andersson2023}. For an overview, we refer to the recent survey \cite{chessari2023numerical}. 

Motivated by their outstanding empirical performance, serious efforts have been made in order to establish convergence guarantees for such deep BSDE algorithms. Originally the pioneering paper of Han and Long \cite{hanlong2020} managed to show a \emph{posteriori bound} which depends on the objective loss functional of the machine learning algorithm. Their result was later extended in \cite{jiang_convergence_2021} to the case of non-Lipschitz continuous drift coefficients, and in \cite{huang_convergence_2025} to the vector-valued framework in the context of stochastic optimal control. These works all have in common that they relied on the assumption of a narrower class of FBSDEs, in fact they considered a special case of \eqref{eq:FBSDE} where only the $Y$ process enters the dynamics of $X$, and not $Z$. Consequently, these convergence results were inapplicable in the context a wide range of stochastic control problems, for instance formulated through the dynamic programming principle \cite{pham2009continuous}, where coupling occurs in $Z$. In particular, these works could not provide a theoretical explanation for the phenomena observed in \cite{andersson2023}, where they found empirical evidence for the non-convergence of the deep BSDE method for FBSDEs stemming from stochastic control. More recently, the authors of \cite{reisinger_posteriori_2024} proved an \emph{a-posteriori} error estimate in the framework of McKean-Vlasov FBSDEs. They give a new result for general, fully-coupled equations even a mean-field term, extending earlier analyses for decoupled equations \cite{bender_posteriori_2013}. Nonetheless, their result is based on a different set of assumptions as the one considered herein, which is concerned with the framework of \cite{hanlong2020}.

The main motivation of the present paper is to extend the convergence result of Han and Long \cite{hanlong2020} to drift coefficients in \eqref{eq:FBSDE} that depend on the $Z$ process as well. The main challenge is to handle the error estimate of the $X$ process with extra $Z$ coupling, which we control by the new estimates established in lemma 
\ref{lem:estimate2}. This enables us to derive our main result, stated in theorem \ref{thm:estimate2}, which is a posteriori error estimate similar to \cite{hanlong2020}. 
In particular, our work enjoys several relevant features: 
\begin{itemize}
    \item we recover the results of \cite{hanlong2020} in the limit case of no $Z$ coupling;
    \item our result is applicable for a general class of FBSDEs, including, but not limited to, the ones obtained for stochastic control problems stemming either from dynamic programming \cite{andersson2023} or the stochastic maximum principle \cite{ji2022, huang_convergence_2025}, due to the coupling of not just $Y$ but also $Z$ in the forward process;
    \item given a particular FBSDE, we can check whether or not it satisfies the convergence conditions through a directly verifiable approach.
\end{itemize}
The paper is organized as follows. In section \ref{section:deep_bsde}, we give the discrete time approximation scheme of the deep BSDE algorithm with $Z$ coupling in the drift $b$. Section \ref{section:convergence} contains our main result stated in theorem \ref{thm:estimate2}. Thereafter the abstract sufficient conditions of convergence are analyzed in section \ref{section:interpretation}. In particular, we show that the assumptions of theorem \ref{thm:estimate2} hold under heuristic interpretations such as \emph{weak coupling} or \emph{small time duration}. Additionally, we get earlier convergence results \cite{hanlong2020} as a limit case of our more general theory. Finally, we demonstrate our theoretical contributions by several numerical experiments on high-dimensional FBSDEs in section \ref{section:numerical}. These simulations confirm and showcase our theoretical findings.

\section{The deep BSDE algorithm}\label{section:deep_bsde}
In this section we formulate the deep BSDE algorithm for FBSDEs as in \eqref{eq:FBSDE}, naturally extending \cite{han2018, e_deep_2017, hanlong2020} to the framework of $Z$ coupling in $b$. For the rest of the paper, we denote the Frobenius norm by $\|x\|$ for any $x\in\mathbb{R}^{i\times j}$, not to be confused with the matrix $2$-norm $\|x\|_2$. Without loss of generality, we work with an equidistant time partition $\pi\coloneqq \{t_i, i=0, \dots, N\vert 0=t_0<t_1<\cdots<t_N=T\}$ with $h=T/N$ and study the following discretization
\begin{subequations}
    \begin{align}
        \inf_{\varphi_0 \in \mathcal{N}_0^Y(\theta_0^Y), \zeta_i \in  \mathcal{N}_i^Z(\theta^Z_i)}  &\mathbb{E} \left[\left\| g\left(X_{t_{N}}^\pi\right) - Y_{t_{N}}^\pi \right\|^2\right], \label{euler_objective}  \\
        &\text{s.t.}
\left\{
\begin{aligned}
&X_0^\pi= x_0, \\
&Y_0^\pi= \varphi_0(x_0; \theta^{Y}_{0}), \\
&X_{t_{i+1}}^\pi = \begin{aligned}[t]
    X_{t_i}^\pi &+ b\left(t_i, X_{t_i}^\pi, Y_{t_i}^\pi, Z_{t_i}^\pi \right) h+ \sigma\left(t_i, X_{t_i}^\pi, Y_{t_i}^\pi \right) \Delta W_i,
\end{aligned} \\
& Z_{t_i}^\pi = \zeta_i(X_{t_i}^{\pi}; \theta^{Z}_{i}) , \\
& Y_{t_{i+1}}^\pi = Y_{t_i}^\pi - f\left(t_i, X_{t_i}^\pi, Y_{t_i}^\pi, Z_{t_i}^\pi\right) h +  Z_{t_i}^\pi  \Delta W_i,
\end{aligned}
\right.\label{eq:euler}
    \end{align}
\end{subequations}
for $i=0, \dots, N-1$, where we put $\Delta W_i\coloneqq W_{t_{i+1}}-W_{t_i}$. In doing so, the numerical solution of a coupled FBSDE \eqref{eq:FBSDE} is reformulated into a stochastic optimization problem consisting of the minimization of an objective functional \eqref{euler_objective} subject to the Euler-Maruyama discretization \eqref{eq:euler}. As in the continuous limit the loss functional \eqref{euler_objective} attains $0$ at the unique solution triple $\{(X_t, Y_t, Z_t)\}_{0\leq t\leq T}$ of \eqref{eq:FBSDE} while also satisfying \eqref{eq:euler}, it is expected that for sufficiently large $N$ and sufficiently wide function spaces $\mathcal{N}_0^Y, \mathcal{N}_i^Z$, the solution of \eqref{euler_objective}-\eqref{eq:euler} is an accurate discrete time approximation of \eqref{eq:FBSDE}. Motivated by universal approximation arguments, see e.g. \cite{hornik_approximation_1991},
we set $\mathcal{N}_0^Y(\theta_0^Y), \mathcal{N}_i^Z(\theta_i^Z)$ to be spaces of deep neural networks parametrized by $\theta_0^Y, \theta_i^Z$ for $i=0, \dots, N-1$.
Subsequently, the goal of the deep BSDE method is to solve this non-linear, constrained optimization problem through the training of deep neural networks. Hence, we seek to find $\varphi_0(x_0; \theta_0^Y)\in \mathcal{N}_0^Y$ and $\zeta_i(X_{t_i}^\pi; \theta_i^Z)\in\mathcal{N}_i^Z$ that approximate $Y_0$ and $Z_{t_i}$ sufficiently well.
The resulting pseudo-code for the complete deep BSDE method is collected in algorithm \ref{algorithm}, its implementation is discussed in section \ref{section:numerical}.
\begin{algorithm}
  \caption{Deep BSDE algorithm}
  \begin{algorithmic}[1] 
    \State \textbf{Input:} Initial parameters $(\theta^{Y, (0)}_0, \theta^{Z, (0)}_i)$, learning rate $\eta$; batch size $M$; number of iterations $K$.
    \State \textbf{Data:} Simulated Brownian increments $\{\Delta W_{t_i}^{(k)}\}_{0\leq i\leq N-1, 1\leq k\leq K}$
    \State \textbf{Output:} Discrete time approximations $\{(\hat{X}_{t_i}^\pi, \hat{Y}_{t_i}^\pi, \hat{Z}_{t_i}^\pi)\}_{i=0, \dots, N}$
    \For{$k = 1$ to $K$}\Comment{Euler-Maruyama \eqref{eq:euler}}
        \State $X_{t_0}^{\pi, (k)} = x_0$, $Y_{t_0}^{\pi, (k)} = \varphi_0(x_0; \theta^{Y, (k-1)}_{0}) $ 
        \For{$i = 0$ to $N-1$}
        \State $ Z_{t_i}^{\pi, (k)} = \zeta_i(X_{t_i}^{\pi, (k)}; \theta^{Z, (k-1)}_{i}) $ 
        \State $ X_{t_{i+1}}^{\pi, (k)} =  X_{t_i}^{\pi, (k)} + b(t_i, X_{t_i}^{\pi, (k)}, Y_{t_i}^{\pi, (k)}, Z_{t_i}^{\pi, (k)}) h  + \sigma (t_i, X_{t_i}^{\pi, (k)}, Y_{t_i}^{\pi, (k)}) \Delta W_{t_i}^{(k)} $
        \State $ Y_{t_{i+1}}^{\pi, (k)} = Y_{t_i}^{\pi, (k)} -  f(t_i, X_{t_i}^{\pi, (k)}, Y_{t_i}^{\pi, (k)}, Z_{t_i}^{\pi, (k)}) h + Z_{t_i}^{\pi, (k)} \Delta W_{t_i}^{(k)} $ 
        \EndFor
    \State $\text{Loss} = \frac{1}{M} \sum_{j=1}^M  \| g( X_{t_N}^{\pi, (k)}) - Y_{t_N}^{\pi, (k)} \|^2 $\Comment{empirical \eqref{euler_objective}}
    \State $  (\theta^{Y, (k)}_0, \theta^{Z, (k)}_0,\dots, \theta^{Z, (k)}_{N-1}) =  (\theta^{Y, (k-1)}_0, \theta^{Z, (k-1)}_0,\ldots, \theta^{Z, (k-1)}_{N-1}) - \eta \nabla \text{Loss} $ \Comment{SGD}
    \EndFor
    \State $(\hat{X}_i^\pi, \hat{Y}_{t_{i}}^\pi, \hat{Z}_{t_{i}}^\pi)=(X_{t_{i}}^{\pi, (K+1)}, Y_{t_{i}}^{\pi, (K+1)}, Z_{t_{i}}^{\pi, (K+1)}),\quad i=0, \dots, N-1$
  \end{algorithmic}
  \label{algorithm}
\end{algorithm}
\section{Convergence analysis}\label{section:convergence}

This section is dedicated to our generalized convergence analysis for the deep BSDE method reviewed in section \ref{section:deep_bsde} and the discrete scheme \eqref{eq:euler} for \eqref{eq:FBSDE}. In particular, we will show that the approximation errors of the numerical solution to the FBSDE are bounded by the simulation error of the objective function corresponding to \eqref{euler_objective} which could be arbitrarily small due to the universal approximation theorem.
The convergence analysis follows a similar strategy as that of \cite{hanlong2020}.
We first introduce the standing assumptions and review some useful result. Throughout the paper we use the notation $\mathbb{E}_i[\cdot] \coloneqq \mathbb{E}[\cdot |\mathcal{F}_i]$.
\begin{assumption}\label{assumption:kb_kf}
There exist constants $k^b$ and $k^f$, that are possibly negative, such that
\begin{equation}
\begin{aligned}
\left( b(t, x_1, y, z) - b(t, x_2, y, z) \right)^\top \Delta x 
& \leq k^b \|\Delta x\|^2, \\
\left( f(t, x, y_1, z)- f(t, x, y_2, z) \right)^\top \Delta y 
& \leq k^f \|\Delta y\|^2.
\end{aligned}
\end{equation}
\end{assumption}
\begin{assumption}\label{assumption:lipschitz}
$b, \sigma, f, g$ are uniformly Lipschitz continuous with respect to $(x, y, z)$. In particular, there are non-negative constants such that
\begin{equation}
\begin{aligned}
\left\|b(t, x_1, y_1, z_1) - b(t, x_2, y_2, z_2)\right\|^2 
& \leq L^{b}_x \|\Delta x\|^2 + L^{b}_y \|\Delta y\|^2 + L^{b}_z \|\Delta z\|^2,  \\
\left\|\sigma(t, x_1, y_1) - \sigma(t, x_2, y_2)\right\|^2 
& \leq L^{\sigma}_x \|\Delta x\|^2 + L^{\sigma}_y \|\Delta y\|^2, \\
\left\| f\left(t, x_1, y_1, z_1\right) - f\left(t, x_2, y_2, z_2\right)\right\|^2 
& \leq L^{f}_x \|\Delta x\|^2+ L^{f}_y \|\Delta y\|^2 + L^{f}_z \|\Delta z\|^2, \\
\left\| g (x_1) - g(x_2)\right\|^2 & \leq  L^{g}_x \|\Delta x\|^2.
\end{aligned}
\end{equation}
\end{assumption}
\begin{assumption} \label{assumption:bounded}
$g(0), b(t,0,0,0), f(t,0,0,0)$ and $\sigma(t,0,0)$ are bounded for $t\in[0, T]$.
\end{assumption}
Notice that assumption \ref{assumption:lipschitz} implies \ref{assumption:kb_kf} with $k^b, k^f\geq 0$. The reason for allowing for negativity shall be made clear by the forthcoming convergence result, see points \eqref{interpretation:k_b<0}, \eqref{interpretation:k_f<0} in section \ref{section:interpretation} below. For convenience, we use $\mathscr{L}$ to denote the set of all constants mentioned above and assume $L$ is the upper bound of $\mathscr{L}$.

Next, we introduce the following system of quasi-linear parabolic PDEs associated with FBSDE \eqref{eq:FBSDE}, 
\begin{equation}\label{eq:quasi-linear-pde}
\left\{
\begin{aligned}
& \partial_t \nu^i  +\frac{1}{2} \partial_{xx} \nu^i: \sigma \sigma^{\top}(t, x, \nu)+\partial_x \nu^i b\left(t, x, \nu, \partial_x \nu \sigma(t, x, \nu)\right)+ f^i\left(t, x, \nu, \partial_x \nu \sigma(t, x, \nu)\right)=0,\\
& \nu(T, x) = g(x),\quad \forall i=1, \cdots, q,
\end{aligned}
\right.
\end{equation}
The following assumption is needed in order to guarantee convergence of the implicit Euler-Maruyama scheme in theorem \ref{convergence_im}.
\begin{assumption}\label{assumption:pde}
The PDE \eqref{eq:quasi-linear-pde} has a classical solution $\nu$ with bounded derivatives $\partial_x \nu$ and $\partial_{xx}^2 \nu$, and $\sigma$ is bounded.
\end{assumption}
The non-linear Feynman-Kac lemma, stated below, establishes the connection between \eqref{eq:quasi-linear-pde} and \eqref{eq:FBSDE}.
\begin{theorem}[Feynman-Kac]\label{uniquesol} Under assumptions \ref{assumption:lipschitz}, \ref{assumption:bounded} and \ref{assumption:pde}, the FBSDE \eqref{eq:FBSDE} has a unique solution $(X, Y, Z)$, and it holds that for $t\in [0, T]$,
\begin{equation}\label{eq: decouple}
Y_t=\nu\left(t, X_t\right), \quad Z_t=\partial_x \nu\left(t, X_t\right) \sigma\left(t, X_t, \nu\left(t, X_t\right)\right).
\end{equation}
\end{theorem}

\begin{remark} The proof of this theorem can be found in \cite[pp. 185-186]{zhang2017backward}. Similar to other numerical methods for coupled FBSDEs, we use this theorem to decouple the original FBSDE \eqref{eq:FBSDE} in order to be able to exploit standard results from the decoupled FBSDE literature. 
\end{remark}
In addition to the assumptions above, we need Hölder-continuity in time for the convergence of the implicit scheme for \eqref{eq:FBSDE}, as stated below.
\begin{assumption}\label{assumption:holder}
$b, \sigma, f$ in \eqref{eq:FBSDE} are uniformly Hölder-$\frac{1}{2}$-continuous with respect to $t$. 
\end{assumption}

Our main result in theorem \ref{thm:estimate2} is concerned with an a posteriori error estimate for \emph{discrete time approximations} of the continuous FBSDE in \eqref{eq:FBSDE}. A necessary ingredient in establishing this is to show that an appropriate discretization such as \eqref{implicit_scheme} below, converges in the number of time steps. This is given in the following theorem.
\begin{theorem}[Convergence of the implicit scheme]\label{convergence_im} Suppose assumptions \ref{assumption:lipschitz}, \ref{assumption:bounded}, \ref{assumption:pde} and \ref{assumption:holder}, and let a suitable set of monotonicity conditions such as in \cite{bender2008time} or \cite{reisinger_posteriori_2024} hold. Then for a sufficiently small $h$, the following discrete-time equation $(0 \leq i \leq N-1)$
\begin{equation}\label{implicit_scheme}
\left\{
\begin{aligned}
&\bar{X}_0^\pi= x_0,\\
&\bar{X}_{t_{i+1}}^\pi = \bar{X}_{t_i}^\pi + b\left(t_i, \bar{X}_{t_i}^\pi, \bar{Y}_{t_i}^\pi,  \bar{Z}_{t_i}^\pi \right) h + \sigma\left(t_i, \bar{X}_{t_i}^\pi, \bar{Y}_{t_i}^\pi\right) \Delta W_i, \\
&\bar{Y}_T^\pi = g\left(\bar{X}_T^\pi\right), \\
&\bar{Z}_{t_i}^\pi = \frac{1}{h} \mathbb{E}_i \left[ \bar{Y}_{t_{i+1}}^\pi \Delta W_i^\top \right], \\
&\bar{Y}_{t_i}^\pi = \mathbb{E}_i \left[ \bar{Y}_{t_{i+1}}^\pi + f\left(t_i, \bar{X}_{t_i}^\pi, \bar{Y}_{t_i}^\pi, \bar{Z}_{t_i}^\pi\right) h \right],
\end{aligned}
\right.
\end{equation}
has a solution $\{(\bar{X}_{t_i}^\pi, \bar{Y}_{t_i}^\pi, \bar{Z}_{t_i}^\pi)\}_{i=0, \dots, N}$, such that $\bar{X}_{t_i}^\pi \in L^2\left(\Omega, \mathcal{F}_{t_i}, \mathbb{P}\right)$ and
\begin{equation}\label{eq:implicit_scheme:convergence:estimate}
\sup_{t \in[0, T]} \left( \mathbb{E} [ \left\|X_t-\bar{X}_t^\pi\right\|^2 ] + \mathbb{E} [\left\|Y_t-\bar{Y}_t^\pi\right\|^2 ]  \right) + \int_0^T \mathbb{E} [ \left\|Z_t-\bar{Z}_t^\pi\right\|^2 ] \mathrm{d}t
\leq C \left(1+\mathbb{E}\|x_0\|^2\right) h,
\end{equation}
with $\bar{X}_t^\pi\coloneqq\bar{X}_{t_i}^\pi, \bar{Y}_t^\pi\coloneqq\bar{Y}_{t_i}^\pi, \bar{Z}_t^\pi\coloneqq\bar{Z}_{t_i}^\pi$ for $t \in\left[t_i, t_{i+1}\right)$, where $C$ is a constant depending on $\mathscr{L}$ and $T$.
\end{theorem}

\begin{remark} We emphasize that the convergence of the discrete time approximation scheme in \eqref{implicit_scheme} is only a necessary ingredient in the last step of our main result in theorem \ref{thm:estimate2}. The objective of the present paper is to prove an a posteriori estimate -- see \eqref{eq:main_theorem:estimate} below -- given a deep BSDE approximation, and not to analyze the convergence of an abstract time discretization. In fact, regardless of the assumptions under which theorem \ref{convergence_im} is stated, as long as the implicit scheme admits an estimate such as \eqref{eq:implicit_scheme:convergence:estimate}, the conclusions of our main result remain the same. Such an estimate can be established by several different approaches in the literature. For instance, it can be shown that the weak and monotonicity conditions in \cite{bender2008time} can be extended to our setting where the drift function $b$ has an extra argument $Z$. Alternatively, one could employ the convergence result of Reisinger et al. in \cite{reisinger_posteriori_2024} for the implicit Euler scheme in the framework of McKean-Vlasov FBSDEs, which would lead to a different set of monotonicity assumptions. For an overview on the literature of time discretization results, we refer to \cite{zhang2017backward}.
\end{remark}

Recall the classical Euler scheme in \eqref{eq:euler}. Taking conditional expectations of the discrete equation of $Y_{t_{i+1}}^\pi$, and of the same equation multiplied by $(\Delta W_i)^\top$, we obtain a formulation that does not include the objective functional \eqref{euler_objective}, i.e.,
\begin{equation}\label{formulation3}
\left\{
\begin{aligned}
X_0^\pi &= x_0, \\
X_{t_{i+1}}^\pi &=X_{t_i}^\pi + b\left(t_i, X_{t_i}^\pi, Y_{t_i}^\pi, Z_{t_i}^\pi \right) h+\sigma\left(t_i, X_{t_i}^\pi, Y_{t_i}^\pi\right) \Delta W_i, \\
Z_{t_i}^\pi &= \frac{1}{h} \mathbb{E}_i \left[ Y_{t_{i+1}}^\pi \Delta W_i^\top  \right], \\
Y_{t_i}^\pi &= \mathbb{E}_i \left[ Y_{t_{i+1}}^\pi + f\left(t_i, X_{t_i}^\pi, Y_{t_i}^\pi,  Z_{t_i}^\pi\right) h  \right] .
\end{aligned}\right.
\end{equation}
With formulation \eqref{formulation3} in hand, we can derive the following apriori estimate bounding the difference between two solutions of it.
\begin{lemma}\label{lem:estimate1} 
For $j=1,2$, suppose $\left(\left\{X_{t_i}^{\pi, j}\right\}_{0 \leq i \leq N},\left\{Y_{t_i}^{\pi, j}\right\}_{0 \leq i \leq N},\left\{Z_{t_i}^{\pi, j}\right\}_{0 \leq i \leq N-1}\right)$ are two solutions of \eqref{formulation3}, with $X_{t_i}^{\pi, j}, Y_{t_i}^{\pi, j} \in L^2\left(\Omega, \mathcal{F}_{t_i}, \mathbb{P}\right), 0 \leq i \leq N$. For any $\lambda_1 > 0, \lambda_2 > L^f_z$, and sufficiently small $h$, denote
\begin{align}\label{def:K1,K2,K3,K4,C1}
    K_1 &\coloneqq 2 k^b+\lambda_1+L^\sigma_x+ L^b_x h,\quad K_2 \coloneqq \left(\lambda_1^{-1}+h\right) L^b_y + L^\sigma_y,\quad 
    K_3 \coloneqq -\frac{\ln \left( 1-\left(2 k^f+\lambda_2\right) h\right) }{h},\\
    K_4 &\coloneqq \frac{L^f_x}{\left( 1-\left(2 k^f+\lambda_2\right) h\right) \lambda_2},\quad C_1 \coloneqq L^b_z ( h + \lambda_1^{-1}).
\end{align}
Let $\delta X_i\coloneqq X_{t_i}^{\pi, 1}-X_{t_i}^{\pi, 2}$, $\delta Y_i\coloneqq Y_{t_i}^{\pi, 1}-Y_{t_i}^{\pi, 2}$, $\delta Z_i\coloneqq Z_{t_i}^{\pi, 1}-Z_{t_i}^{\pi, 2}$, then we have, for $0 \leq n \leq N$
\begin{align}
    \mathbb{E} \left[ \left\|\delta X_n  \right\|^2 \right]  &\leq  K_2 h \sum_{i=0}^{n-1} e^{K_1(n-i-1) h} \mathbb{E} \left[ \left\| \delta Y_i \right\|^2 \right]  + C_1 h \sum_{i=0}^{n-1} e^{K_1(n-i-1) h} \mathbb{E} \left[ \left\|\delta Z_i\right\|^2 \right],\label{eq:lemma1:x}\\
\mathbb{E} \left[ \left\| \delta Y_i \right\|^2 \right] &\leq e^{K_3(N-n) h} \mathbb{E} \left[ \left\| \delta Y_N \right\|^2 \right] + K_4 \sum_{i=n}^{N-1} e^{K_3(i-n) h} \mathbb{E} \left[ \left\| \delta X_i \right\|^2 \right] h .\label{eq:lemma1:y}
\end{align}
\end{lemma} 

\begin{proof} Let us define
\begin{align}
    \delta b_i & \coloneqq b\left(t_i, X_{t_i}^{\pi, 1}, Y_{t_i}^{\pi, 1}, Z_{t_i}^{\pi, 1} \right) - b\big(t_i, X_{t_i}^{\pi, 2}, Y_{t_i}^{\pi, 2}, Z_{t_i}^{\pi, 2} \big), \delta \sigma_i \coloneqq \sigma\left(t_i, X_{t_i}^{\pi, 1}, Y_{t_i}^{\pi, 1}\right)-\sigma\left(t_i, X_{t_i}^{\pi, 2}, Y_{t_i}^{\pi, 2}\right), \\
    \delta f_i &\coloneqq f\left(t_i, X_{t_i}^{\pi, 1}, Y_{t_i}^{\pi, 1}, Z_{t_i}^{\pi, 1}\right) - f\left(t_i, X_{t_i}^{\pi, 2}, Y_{t_i}^{\pi, 2}, Z_{t_i}^{\pi, 2}\right).
\end{align}
Then we have
\begin{align}
    \delta X_{i+1} &= \delta X_i + \delta b_i h + \delta \sigma_i \Delta W_i,\label{deltaX}\\
    \delta Y_i &= \mathbb{E}_i \left[ \delta Y_{i+1}+\delta f_i h \right],\label{deltaP}
\end{align}
and \eqref{formulation3} also gives
\begin{equation}\label{deltaQ}
\delta Z_i = \frac{1}{h} \mathbb{E}_i \left[ \delta Y_{i+1} \Delta W_i^\top \right].
\end{equation}
By the martingale representation theorem, there exists an $\mathcal{F}_t$-adapted square-integrable process $\left\{\delta Z_t\right\}_{t_i \leq t \leq t_{i+1}}$ such that
\begin{equation}\label{martingale_P1}
\delta Y_{i+1} = \mathbb{E}_i \left[  \delta Y_{i+1} \right] +  \int_{t_i}^{t_{i+1}} \delta Z_t \mathrm{d}W_t,
\end{equation}
which, together with \eqref{deltaP}, implies
\begin{equation}\label{martingale_P2}
\delta Y_{i+1} = \delta Y_i-\delta f_i h + \int_{t_i}^{t_{i+1}} \delta Z_t \mathrm{d}W_t.
\end{equation}
From \eqref{deltaX} and \eqref{martingale_P2}, noting that $\delta X_i, \delta Y_i, \delta b_i, \delta \sigma_i$ and $\delta f_i$ are all $\mathcal{F}_{t_i}$ measurable, and $\mathbb{E}_i [\Delta W_i ]=0$, $\mathbb{E}_i [ \int_{t_i}^{t_{i+1}} Z_t \mathrm{d}W_t ]=0$, we have
\begin{align}
    \mathbb{E} \left[ \left\| \delta X_{i+1} \right\|^2\right] 
& = \mathbb{E} \left[ \left\| \delta X_i+\delta b_i h\right\|^2 \right] + h \mathbb{E} \left[ \left\|\delta \sigma_i\right\|^2 \right],\\
    \mathbb{E} \left[ \left\| \delta Y_{i+1} \right\|^2 \right] & = \mathbb{E} \left[ \left\| \delta Y_i-\delta f_i h\right\|^2 \right] + \int_{t_i}^{t_{i+1}} \mathbb{E} \left[ \left\| \delta Z_t\right\|^2 \right] \mathrm{d}t,
\end{align}
where we also used a Fubini argument. We proceed in steps, controlling each of the terms above.
\begin{steps}[wide, labelwidth=0pt, labelindent=0pt]
\item \label{lemma2:step:2}\textit{Estimate for $\delta X_n$.}
By assumptions \ref{assumption:kb_kf}, \ref{assumption:lipschitz}, and the root-mean-square and geometric mean inequality (RMS-GM inequality), we have that for any $\lambda_1>0$
\begin{equation}
\mathbb{E} \left[ \left\| \delta X_{i+1} \right\|^2 \right] \begin{aligned}[t]
  &=\begin{aligned}[t]
      & \mathbb{E} \left[\left\|\delta X_i\right\|^2\right] + \mathbb{E} \left[\left\|\delta b_i\right\|^2\right] h^2 + h \mathbb{E} \left[\left\|\delta \sigma_i\right\|^2\right] \\
& +2 h \mathbb{E} \left[\left(b\left(t_i, X_{t_i}^{\pi, 1}, Y_{t_i}^{\pi, 1}, Z_{t_i}^{\pi, 1} \right) - b\left(t_i, X_{t_i}^{\pi, 2}, Y_{t_i}^{\pi, 1}, Z_{t_i}^{\pi, 1} \right)\right)^\top \delta X_i\right] \\
& +2 h \mathbb{E} \left[\left(b\left(t_i, X_{t_i}^{\pi, 2}, Y_{t_i}^{\pi, 1}, Z_{t_i}^{\pi, 1} \right)-b\left(t_i, X_{t_i}^{\pi, 2}, Y_{t_i}^{\pi, 2}, Z_{t_i}^{\pi, 2} \right)\right)^\top \delta X_i\right] \\
  \end{aligned}\\
    &\leq\begin{aligned}[t]
    &\left(1+\left(2 k^b+\lambda_1+L^\sigma_x + L^b_x h\right) h\right) \mathbb{E}\left[\left\|\delta X_i\right\|^2\right] \\ &+ \left(\left(\lambda_1^{-1}+h\right) L^b_y+ L^\sigma_y\right) \mathbb{E}\left[\left\|\delta Y_i\right\|^2\right] h + \left( L^b_z h + \lambda_1^{-1} L^b_z \right) \mathbb{E} \left[\left\|\delta Z_i\right\|^2\right] h.
\end{aligned}  
\end{aligned}
\end{equation}
Recalling the definition of $C_1, K_1, K_2$ from \eqref{def:K1,K2,K3,K4,C1}, we subsequently gather
\begin{equation}
\begin{aligned}
\mathbb{E} \left[\left\| \delta X_{i+1} \right\|^2\right] 
\leq (1+K_1 h) \mathbb{E} \left[\left\|\delta X_i\right\|^2\right] + K_2 h \mathbb{E} \left[ \left\|\delta Y_i\right\|^2 \right] + C_1 h \mathbb{E} \left[\left\|\delta Z_{i}\right\|^2\right] .
\end{aligned}
\end{equation}
Notice that $\mathbb{E} \left[\left\|\delta X_0\right\|^2 \right]=0 $, and thus by induction, we have that for any $1 \leq n \leq N$,
\begin{equation}
\begin{aligned}
\mathbb{E} \left[ \left\|\delta X_n\right\|^2 \right]
\leq & \prod_{i=0}^{n-1} (1+K_1 h) \mathbb{E} \left[ \left\|\delta X_0\right\|^2\right] + \sum_{i=0}^{n-1} (1+K_1 h)^{n-1-i} K_2 \mathbb{E} \left[\left\|\delta Y_i\right\|^2\right] h  \\
& + \sum_{i=0}^{n-1} (1+K_1 h)^{n-1-i} C_1 \mathbb{E} \left[\left\|\delta Z_i\right\|^2\right] h \\
\leq & K_2 h \sum_{i=0}^{n-1} e^{K_1(n-i-1) h} \mathbb{E} \left[\left\|\delta Y_i\right\|^2 \right] + C_1 h \sum_{i=0}^{n-1} e^{K_1(n-i-1) h} \mathbb{E} \left[ \left\|\delta Z_i\right\|^2\right],
\end{aligned}
\end{equation}
where we used the inequality $(1+x)\leq e^x$, $\forall x\in\R$.
We remark that due to the coupling of $Z$ in the drift, the last term of the right-hand side above is not present in \cite{hanlong2020}.

\item \label{lemma2:step:1}\textit{Estimate for $\delta Y_n$.}
We employ a similar approach as in step 1. Using assumption \ref{assumption:lipschitz} and the RMS-GM inequality, we obtain for any $\lambda_2>0$, 
\begin{align}\label{lemma2:delta_p_i}
     \mathbb{E} \left[ \left\|\delta Y_{i+1}\right\|^2 \right]\begin{aligned}[t]
&\geq\begin{aligned}[t]
    & \mathbb{E} \left[ \left\|\delta Y_i\right\|^2\right] + \int_{t_i}^{t_{i+1}} \mathbb{E} \left[ \left\|\delta Z_t\right\|^2\right] \mathrm{d}t \\
& -2 h \mathbb{E} \left[\left( f\left(t_i, X_i^{1, \pi}, Y_i^{1, \pi}, Z_i^{1, \pi}\right)-f \left(t_i, X_i^{1, \pi}, Y_i^{2, \pi}, Z_i^{1, \pi}\right)\right)^\top \delta Y_i\right] \\
& -2 h \mathbb{E} \left[\left( f\left(t_i, X_i^{1, \pi}, Y_i^{2, \pi}, Z_i^{1, \pi}\right)-f \left(t_i, X_i^{2, \pi}, Y_i^{2, \pi}, Z_i^{2, \pi}\right)\right)^\top \delta Y_i\right] 
\end{aligned}\\
&\geq  \begin{aligned}[t]
    & \mathbb{E} \left[ \left\|\delta Y_i\right\|^2 \right] + \int_{t_i}^{t_{i+1}} \mathbb{E} \left[ \left\|\delta Z_t\right\|^2\right] \mathrm{d}t  -  2 k^f h \mathbb{E} \left[ \left\|\delta Y_i\right\|^2\right] \\
& - \left( \lambda_2  \mathbb{E} \left[ \left\|\delta Y_i\right\|^2 \right] + \lambda_2^{-1} \left( L^f_x \mathbb{E} \left[ \left\|\delta X_i\right\|^2\right] + L^f_z \mathbb{E}\left[ \left\|\delta Z_i\right\|^2\right] \right)\right) h.
\end{aligned}
\end{aligned}
\end{align}
To deal with the integral term in the last inequality, we derive the following relation via Ito's isometry, \eqref{martingale_P2} and \eqref{deltaQ}
\begin{equation}
\delta Z_i = \frac{1}{h} \mathbb{E}_i \left[ \int_{t_i}^{t_{i+1}} \delta Z_t \mathrm{d}t \right] .
\end{equation}
Then, by the Jensen- and Cauchy-Schwartz inequalities and the Fubini theorem, we derive a lower bound for the integral term
\begin{equation}\label{lemma2:integral_term}
\begin{aligned}
\mathbb{E} \left[ \left\| \delta Z_i\right\|^2\right] h 
& = \sum_{j=1}^q \sum_{k=1}^m \mathbb{E}\left[ \left(\delta Z_i\right)_{j,k}^2 \right] h = \sum_{j=1}^q \sum_{k=1}^m \frac{1}{h} \mathbb{E} \left[\left( \mathbb{E}_i\left[\int_{t_i}^{t_{i+1}} \left(\delta Z_t\right)_{j,k}  \mathrm{d}t \right] \right)^2  \right]\\
& \leq \sum_{j=1}^q \sum_{k=1}^m \frac{1}{h} \mathbb{E}\left[\left(\int_{t_i}^{t_{i+1}} \left(\delta Z_t\right)_{j,k} \mathrm{d}t \right)^2\right] \\
& \leq \sum_{j=1}^q \sum_{k=1}^m \int_{t_i}^{t_{i+1}} \mathbb{E}\left[ \left(\delta Z_t\right)_{j,k}^2 \right] \mathrm{d}t =\int_{t_i}^{t_{i+1}} \mathbb{E}\left[\left\|\delta Z_t\right\|^2\right] \mathrm{d}t,
\end{aligned}
\end{equation}
where $(\cdot)_{j,k}$ denotes the $(j,k)$-entry of the matrix. 
Combining \eqref{lemma2:delta_p_i} with \eqref{lemma2:integral_term} gives
\begin{equation}\label{deltaP_2}
\mathbb{E} \left[ \left\|\delta Y_{i+1}\right\|^2 \right]
\geq \begin{aligned}[t]&\left(1-\left(2 k^f +\lambda_2\right) h\right) \mathbb{E} \left[ \left\|\delta Y_i\right\|^2 \right] + \left(1 - L^f_z \lambda_2^{-1}\right) \mathbb{E} \left[\left\|\delta Z_i\right\|^2 \right] h\\
& - L^f_x \lambda_2^{-1} \mathbb{E} \left[\left\|\delta X_i\right\|^2\right] h .
\end{aligned}
\end{equation}
Now, for any $\lambda_2 > L^f_z \geq 0$, and sufficiently small $h$ satisfying $ \left(2 k^f+\lambda_2\right) h<1$, this implies
\begin{equation}
\mathbb{E} \left[ \left\|\delta Y_i\right\|^2 \right]
\leq \left(1-\left(2 k^f +\lambda_2\right) h\right)^{-1} \left(\mathbb{E} \left[ \left\|\delta Y_{i+1}\right\|^2\right] + L^f_x \lambda_2^{-1} \mathbb{E} \left[\left\|\delta X_i\right\|^2\right] h\right) .
\end{equation}
Recalling the definitions of $K_3, K_4$ in \eqref{def:K1,K2,K3,K4,C1}, we subsequently gather by induction that for any $0 \leq n \leq N-1$
\begin{equation}\label{eq:lemma2:delta_pn}
\mathbb{E} \left[\left\|\delta Y_n\right\|^2\right] \leq e^{K_3(N-n) h} \mathbb{E} \left[ \left\|\delta Y_N\right\|^2\right] + K_4 \sum_{i=n}^{N-1} e^{K_3(i-n) h} \mathbb{E} \left[\left\|\delta X_i\right\|^2\right] h.
\end{equation}
We remark that this estimate coincides with the one of \cite[Lemma 1]{hanlong2020}.
\end{steps}
\end{proof}
Due to the coupling of $Z$ in the drift coefficient of the forward diffusion, we need an additional estimate to handle the extra $\mathbb{E}[\|\delta Z_i\|^2]$ term in the estimate for $\mathbb{E}[\|\delta X_i\|^2]$. One of our main contributions is to establish the following lemma for this purpose. 
\begin{lemma}\label{lem:estimate2} Under the setting of lemma \ref{lem:estimate1}, for any $\lambda_3> 2m L^f_z$ and sufficiently small $h$, let us define 
\begin{align}\label{def:C2,C3,C4}
    C_2 \coloneqq 2( (h+\lambda_3^{-1}) L^f_y + \lambda_3),\quad C_3 \coloneqq  2 (h+\lambda_3^{-1} ),\quad
C_4 \coloneqq ( 1- mC_3 L^f_z )^{-1} m.
\end{align}
Then we have $C_4>0$, furthermore, the following estimates also hold
\begin{align}
    \mathbb{E} \left[ \|\delta Y_i\|^2 \right]& \leq  (1+ C_2 h)  \mathbb{E} \left[ \left\| \mathbb{E}_i \left[\delta Y_{i+1} \right] \right\|^2 \right]  + C_3 h L^f_x  \mathbb{E}\left[\|\delta X_i\|^2\right]  + C_3  h L^f_z \mathbb{E}\left[\|\delta Z_i\|^2\right], \label{eq:lemma2:y}\\
    h \sum_{i=0}^{N-1} \mathbb{E}[\|\delta Z_i\|^2]
& \leq C_4 \left( \sum_{i=1}^{N-1} C_2 h \mathbb{E} \left[ \left\| \mathbb{E}_i \left[\delta Y_{i+1} \right]\right\|^2 \right] + C_3 h L^f_x \mathbb{E}\left[\|\delta X_i\|^2\right] \right) + C_4 \mathbb{E}\left[\|\delta Y_N\|^2\right] . \label{eq:lemma2:z}
\end{align}
\end{lemma}
\begin{proof}
We take the squares of both sides of \eqref{deltaP} and use the $\epsilon$-Young inequality to get
\begin{equation}
\left\|\delta Y_i\right\|^2 
\leq(1+\lambda_3 h) \left\|\mathbb{E}_i\left[\delta Y_{i+1}\right]\right\|^2 + (1+(\lambda_3 h)^{-1} ) \left\| h \delta f_i\right\|^2,
\end{equation}
which holds for any $\lambda_3>0$, independent of $h>0$.
Taking expectations on both sides and using the Lipschitz continuity of $f$ established by assumption \ref{assumption:lipschitz} yields
\begin{equation}
\begin{aligned}
\mathbb{E} \left[ \left\|\delta Y_i\right\|^2 \right]
& \leq (1+\lambda_3 h) \mathbb{E}\left[\left\|\mathbb{E}_i\left[\delta Y_{i+1}\right]\right\|^2\right] + \left( h + \lambda_3^{-1} \right) h \mathbb{E}\left[\left\|\delta f_i\right\|^2\right]  \\
& \leq (1+\lambda_3 h) \mathbb{E}\left[\left\|\mathbb{E}_i\left[\delta Y_{i+1}\right]\right\|^2\right] + (h+ \lambda_3^{-1}) h\left( L^f_x \mathbb{E}\left[\|\delta X_i\|^2\right] + L^f_y \mathbb{E}[\|\delta Y_i\|^2 ] \right. \\
& \qquad \left. + L^f_z \mathbb{E}\left[\|\delta Z_i\|^2\right] \right).
\end{aligned}
\end{equation}
Therefore by a rearrangement 
\begin{equation}
\begin{aligned}
(1- (h+\lambda_3^{-1}) h L^f_y ) \mathbb{E} \left[\left\|\delta Y_i\right\|^2\right] 
\leq & (1+\lambda_3 h) \mathbb{E}\left[ \left\| \mathbb{E}_i \left[\delta Y_{i+1} \right]\right\|^2\right] \\
& + (h+ \lambda_3^{-1}) h L^f_x \mathbb{E}\left[ \|\delta X_i\|^2\right] + \left(h+\lambda_3^{-1}\right) h L^f_z \mathbb{E}\left[\|\delta Z_i\|^2\right].
\end{aligned}
\end{equation}
Consequently, for any $\lambda_3>0$ and sufficiently small $h$, we obtain the following estimate, for $i=0, 1, \ldots, N-1$
\begin{equation}\label{estimate_by_taylor}
\begin{aligned}
\mathbb{E} \left[ \left\|\delta Y_i\right\|^2 \right]
\leq & (1+ C_2 h)  \mathbb{E}\left[\left\| \mathbb{E}_i [\delta Y_{i+1} ]\right\|^2\right]  + C_3 h L^f_x \mathbb{E}\left[\|\delta X_i\|^2\right] + C_3  h L^f_z \mathbb{E}\left[\|\delta Z_i\|^2\right],
\end{aligned}
\end{equation}
where we used the definitions in \eqref{def:C2,C3,C4}. This proves \eqref{eq:lemma2:y}.

Next, we derive the estimate for $Z$. Recalling the definition in \eqref{deltaQ}, we get
\begin{equation}
\begin{aligned}
h  \delta Z_i 
& = \mathbb{E}_i\left[\delta Y_{i+1} \Delta W_i^{\top}\right]  =\mathbb{E}_i\left[\left(\delta Y_{i+1}-\mathbb{E}_i\left[\delta Y_{i+1}\right]\right) \Delta W_i^{\top}\right].
\end{aligned}
\end{equation}
Taking the Frobenius norm on both sides and applying the Cauchy–Schwartz inequality then yields\begin{equation}
\begin{aligned}
h \left\|\delta Z_i\right\| 
& = \left\| \mathbb{E}_i\left[ \left(\delta Y_{i+1}-\mathbb{E}_i\left[\delta Y_{i+1}\right]\right) \Delta W_i^{\top}\right] \right\| \\
& \leq \left(\mathbb{E}_i\left[\left\|\delta Y_{i+1}-\mathbb{E}_i\left[\delta Y_{i+1}\right]\right\|^2\right]\right)^{\frac{1}{2}} \left(\mathbb{E}_i\left[ \left\|\Delta W_i^{\top}\right\|^2\right] \right)^{\frac{1}{2}} \\
& =\left(\mathbb{E}_i\left[\left\|\delta Y_{i+1}-\mathbb{E}_i\left[\delta Y_{i+1}\right]\right\|^2\right]\right)^{\frac{1}{2}}  (hm)^{\frac{1}{2}},
\end{aligned}
\end{equation}
which leads to
\begin{equation}
\begin{aligned}
h \mathbb{E}\left[ \|\delta Z_i\|^2 \right]
\leq m \mathbb{E} \left[ \left\|\delta Y_{i+1}-\mathbb{E}_i\left[\delta Y_{i+1}\right]\right\|^2 \right] 
= m \left( \mathbb{E}\left[\|\delta Y_{i+1}\|^2\right] - \mathbb{E}\left[ \left\| \mathbb{E}_i\left[\delta Y_{i+1}\right] \right\|^2 \right] \right).
\end{aligned}
\end{equation}
Summing both sides from $0$ to $N-1$ and using the estimate \eqref{estimate_by_taylor}, we gather
\begin{equation}
\begin{aligned}
h \sum_{i=0}^{N-1} \mathbb{E} \left[ \left\|\delta Z_i\right\|^2 \right]
& \leq m \sum_{i=0}^{N-1} \left( \mathbb{E}\left[\|\delta Y_{i+1}\|^2\right] - \mathbb{E}\left[ \left\| \mathbb{E}_i\left[\delta Y_{i+1}\right] \right\|^2 \right]\right) \\
& = m \sum_{i=1}^{N-1}\left( \mathbb{E}\left[ \|\delta Y_i\|^2\right] - \mathbb{E}\left[\left\|\mathbb{E}_i\left[\delta Y_{i+1}\right]\right\|^2\right] \right) + \mathbb{E}\left[ \|\delta Y_N\|^2\right] - \mathbb{E}\left[\left\|\mathbb{E}_0\left[\delta Y_1\right]\right\|^2\right] \\
& \leq m \sum_{i=1}^{N-1}  \left( \mathbb{E}\left[ \|\delta Y_i\|^2\right]-\mathbb{E}\left[\left\|\mathbb{E}_i\left[\delta Y_{i+1}\right]\right\|^2\right]\right)+ m \mathbb{E}\left[\|\delta Y_N\|^2\right]  \\
& \leq m\sum_{i=1}^{N-1} \left( C_2 h \mathbb{E}\left[\left\| \mathbb{E}_i \left[\delta Y_{i+1} \right] \right\|^2\right] + C_3 h L^f_x \mathbb{E}\left[\left\|\delta X_i\right\|^2\right] + C_3 h L^f_z \mathbb{E}\left[\left\|\delta Z_i\right\|^2\right] \right)  \\
& \qquad  + m \mathbb{E}\left[ \| \delta Y_N \|^2 \right].
\end{aligned}
\end{equation}
Recalling $C_4$ in \eqref{def:C2,C3,C4}, it is easy to check that for any $\lambda_3> 2mL^f_z$ and sufficiently small $h>0$, we have $C_4>0$ and therefore
\begin{equation}
\begin{aligned}
h \sum_{i=0}^{N-1} \mathbb{E} \left[\left\|\delta Z_i\right\|^2\right] 
& \leq C_4 \sum_{i=1}^{N-1} \left(  C_2 h \mathbb{E} \left[ \left\| \mathbb{E}_i [\delta Y_{i+1} ]\right\|^2 \right]  + C_3 h L^f_x \mathbb{E}\left[\|\delta X_i\|^2\right] \right) + C_4 \mathbb{E}\left[ \| \delta Y_N \|^2 \right].
\end{aligned}
\end{equation}
\end{proof}

\begin{remark} As shown in the proofs of lemma \ref{lem:estimate1} and \ref{lem:estimate2}, the constants $C_j$, $j=1,2,3,4$ appear because of the $Z$ coupling in the drift. Conversely, the constants $K_j$, $j=1,2,3,4$ are present even in the less general case of only $Y$ coupling, and they are consistent with \cite{hanlong2020}. In order to emphasize the difference, we denoted these by different letters.
\end{remark}

With these auxiliary results, and particularly lemma \ref{lem:estimate2}, we are ready to state our main result, a posteriori error estimate, generalizing the convergence of the deep BSDE method.
\begin{theorem}[Convergence of the deep BSDE method]\label{thm:estimate2}
Suppose assumptions \ref{assumption:kb_kf}-\ref{assumption:holder} and let the conditions of theorem \ref{convergence_im} hold. Define
\begin{align}
\bar{B} &\coloneqq
	e^{\max(- \bar{K}_1 T, 0)} L^f_x  \bar{C}_1 \bar{C}_4 \bar{C}_3  \frac{e^{\bar{K}_1 T}-1}{\bar{K}_1} + e^{\max(- \bar{K}_1 T, 0)} \bar{C}_1 \bar{C}_4 L^g_x (1+\lambda_4 )  e^{\bar{K}_1 T}, \label{def:B_bar}\\
\bar{A} &\coloneqq \begin{aligned}[t]
		&\left( L^g_x (1+\lambda_4) e^{\left(\bar{K}_1+\bar{K}_3\right) T} + \frac{\bar{K}_4}{\bar{K}_1+\bar{K}_3} \left(e^{\left(\bar{K}_1+\bar{K}_3\right) T}-1\right) \right)\\
		&\times\left(1 - \bar{B} \right)^{-1}\left( \bar{K}_2 \frac{1-e^{-\left(\bar{K}_1+\bar{K}_3\right) T}}{\bar{K}_1+\bar{K}_3} + e^{\max(- \bar{K}_1 T, 0)} \bar{C}_1 \bar{C}_4 \bar{C}_2 \frac{1-e^{-\bar{K}_3 T}}{\bar{K}_3}  \right),
	\end{aligned}
	\label{def:A_bar}
\end{align}
where $\bar{K}_j\coloneqq \lim_{h\to 0} K_j, \bar{C}_j\coloneqq \lim_{h\to 0} C_j$ for $j=1, 2, 3, 4$.
If
\begin{align}\label{thm: condition}
\inf_{\lambda_1>0,\lambda_2>L^f_z, \lambda_3>2mL^f_z, \lambda_4>0} \max (\bar{B}, \bar{A})<1,
\end{align}
then there exists a constant $C>0$, depending only on $\mathbb{E}[\|x_0\|^2], \mathscr{L}, T, \lambda_1$, $\lambda_2$, $\lambda_3$ and $\lambda_4$, such that for sufficiently small $h$, it holds that
\begin{align}\label{eq:main_theorem:estimate}
	\begin{aligned}[b]
		\sup_{t \in[0, T]} 
		\big( \mathbb{E}[\|X_t-\hat{X}_t^\pi\|^2]
		&+ \mathbb{E}[\|Y_t-\hat{Y}_t^\pi\|^2] \big) + \int_0^T \mathbb{E}[\|Z_t-\hat{Z}_t^\pi\|^2] \mathrm{d}t \leq C \big( h + \mathbb{E}[\| g\left(X_T^\pi\right)-Y_T^\pi\|^2] \big),
	\end{aligned}
\end{align}
where $\hat{X}_t^\pi\coloneqq X_{t_i}^\pi, \hat{Y}_t^\pi\coloneqq Y_{t_i}^\pi, \hat{Z}_t^\pi\coloneqq Z_{t_i}^\pi$ for $t \in\left[t_i, t_{i+1}\right)$.
\end{theorem}

\begin{proof}
Let $X_{t_i}^{\pi, 1}=X_{t_i}^\pi, Y_{t_i}^{\pi, 1}=Y_{t_i}^\pi, Z_{t_i}^{\pi, 1}=$ $Z_{t_i}^\pi$ given by the Euler scheme \eqref{eq:euler}, and $X_{t_i}^{\pi, 2}=\bar{X}_{t_i}^\pi, Y_{t_i}^{\pi, 2}=\bar{Y}_{t_i}^\pi, Z_{t_i}^{\pi, 2}=\bar{Z}_{t_i}^\pi$ given by implicit scheme \eqref{implicit_scheme}. Both of them solve \eqref{formulation3}, and therefore we can apply lemma \ref{lem:estimate1} to bound their differences. In what follows, we use the same notations as in the proof of lemma \ref{lem:estimate1}.

First, using the RMS-GM inequality, for any $\lambda_4>0$ we get
\begin{equation}\label{estimate_YN}
\begin{aligned}
\mathbb{E}\big[\|\delta Y_N\|^2\big]
& \equiv \mathbb{E}\big[ \| g(\bar{X}_T^\pi)-Y_T^\pi \|^2\big] \leq  (1+\lambda_4^{-1}) \mathbb{E}\big[ \| g(X_T^\pi)-Y_T^\pi \|^2\big] + L^g_x(1+\lambda_4) \mathbb{E}\big[\|\delta X_N\|^2\big].
\end{aligned}
\end{equation}
Let
\begin{equation}\label{def:cal_XY}
\mathcal{X} \coloneqq \max_{0 \leq n \leq N} e^{-K_1 n h} \mathbb{E} \left[ \left\|\delta X_n\right\|^2 \right], \qquad 
\mathcal{Y} \coloneqq \max_{0 \leq n \leq N} e^{K_3 n h} \mathbb{E} \left[ \left\|\delta Y_n\right\|^2 \right].
\end{equation}
From estimate \eqref{eq:lemma1:y} in lemma \ref{lem:estimate1}, we derive the following by multiplying with $e^{K_3nh}$ on both sides
\begin{equation}
\begin{aligned}
& e^{K_3 nh} \mathbb{E}\left[ \|\delta Y_n\|^2 \right] \\
\leq & e^{K_3 T} \mathbb{E}\left[\|\delta Y_N\|^2\right] + K_4 \sum_{i=n}^{N-1} e^{K_3 ih} \mathbb{E}\left[\|\delta X_i\|^2\right] h \\
\leq & e^{K_3 T} \left( \left(1+\lambda_4^{-1}\right) \mathbb{E}\left[\| g(X_T^\pi)-Y_T^\pi\|^2 \right] + L^g_x (1+\lambda_4) \mathbb{E}\left[\|\delta X_N\|^2\right]  \right) + K_4 \sum_{i=n}^{N-1} e^{K_3 ih} \mathbb{E}\left[\|\delta X_i\|^2\right] h \\
\leq & e^{K_3 T} (1+\lambda_4^{-1}) \mathbb{E}\left[ \| g(X_T^\pi)-Y_T^\pi\|^2\right] + \left( L^g_x (1+\lambda_4) e^{(K_1+K_3) T} + K_4 \sum_{i=n}^{N-1} e^{(K_1+K_3) ih} h \right)  \mathcal{X},
\end{aligned}
\end{equation}
where we used the definition of $\mathcal{X}$ in \eqref{def:cal_XY} and the estimate \eqref{estimate_YN} in the last inequality.
Maximizing over $n$ subsequently yields
\begin{equation}\label{estimate_cal_Y}
\begin{aligned}
\mathcal{Y} & \leq e^{K_3 T} (1+\lambda_4^{-1}) \mathbb{E}\left[\| g(X_T^\pi)-Y_T^\pi\|^2 \right] +  \left(L^g_x (1+\lambda_4) e^{(K_1+K_3) T} + K_4 h \frac{e^{(K_1+K_3) T}-1}{e^{(K_1+K_3) h}-1}\right) \mathcal{X} .
\end{aligned}
\end{equation}

We approach $\delta X_n$ in the same manner, and from \eqref{eq:lemma1:x} collect
\begin{equation}
\begin{aligned}
e^{-K_1 nh} \mathbb{E}\left[\left\|\delta X_n\right\|^2 \right]
& \leq K_2 \sum_{i=0}^{n-1} e^{-K_1 (i+1)h} \mathbb{E}\left[ \left\|\delta Y_i\right\|^2 \right] h + C_1 \sum_{i=0}^{n-1} e^{-K_1 (i+1)h} \mathbb{E}\left[ \left\|\delta Z_i\right\|^2 \right] h \\
& \leq K_2 \mathcal{Y} \sum_{i=0}^{n-1} e^{-K_1 (i+1)h - K_3 ih} h + C_1 \sum_{i=0}^{n-1} e^{-K_1 (i+1)h} \mathbb{E}\left[ \left\|\delta Z_i\right\|^2\right] h.
\end{aligned}
\end{equation}
Additionally, from \eqref{eq:lemma2:y} and \eqref{eq:lemma2:z} we get
\begin{equation}
\begin{aligned}
& C_1 \sum_{i=0}^{n-1} e^{-K_1 h(i+1)} \mathbb{E} \left[\| \delta Z_i\|^2 \right] h \\
\leq & e^{\max(-K_1 T, 0)} C_1 \sum_{i=0}^{N-1} \mathbb{E} \left[\| \delta Z_i\|^2 \right] h  \\
\leq & e^{\max(-K_1 T, 0)} C_1 C_4 \left( \sum_{i=1}^{N-1} C_2 h \mathbb{E}\left[ \left\|\delta Y_{i+1}\right\|^2 \right] + \sum_{i=1}^{N-1} C_3 h L^f_x \mathbb{E} \left[ \| \delta X_i \|^2 \right] +  \mathbb{E} \left[\|\delta Y_N\|^2 \right] \right)  \\
\leq & e^{\max(-K_1 T, 0)}  C_1 C_4 h  \left(C_2 \sum_{i=1}^{N-1} e^{-K_3 (i+1) h}   \mathcal{Y} + C_3  L^f_x  \sum_{i=1}^{N-1} e^{K_1 ih} \mathcal{X} \right)  +  e^{\max(-K_1 T, 0)}  C_1 C_4 \mathbb{E} \left[\|\delta Y_N\|^2 \right]  \\
\leq & e^{\max(-K_1 T, 0)} C_1 C_4 h \left( C_2 e^{-K_3 h} \frac{e^{-K_3 T}-1}{e^{-K_3 h}-1} \mathcal{Y} + C_3  L^f_x  \frac{e^{K_1 T}-1}{e^{K_1 h}-1} \mathcal{X} \right)  +  e^{\max(-K_1 T, 0)} C_1 C_4 \mathbb{E} \left[\|\delta Y_N\|^2 \right],
\end{aligned}
\end{equation}
where we recall the definitions in \eqref{def:K1,K2,K3,K4,C1} and \eqref{def:C2,C3,C4}.
Combining these inequalities and applying estimate \eqref{estimate_YN}, we obtain the following by maximizing over $n$
\begin{equation}\label{estimate_cal_X}
\begin{aligned}
\mathcal{X}
\leq & K_2 \mathcal{Y} h e^{-K_1 h} \frac{e^{-(K_1+K_3) T}-1}{e^{-(K_1+K_3) h}-1} + e^{\max(-K_1 T, 0)} C_1 C_4 h  C_2 e^{-K_3 h} \frac{e^{-K_3 T}-1}{e^{-K_3 h}-1} \mathcal{Y} \\
& + e^{\max(-K_1 T, 0)} C_1 C_4 h C_3  L^f_x  \frac{e^{K_1 T}-1}{e^{K_1 h}-1} \mathcal{X} 
+ e^{\max(-K_1 T, 0)} C_1 C_4 L^g_x (1+\lambda_4) e^{K_1 T} \mathcal{X} \\
& + e^{\max(-K_1 T, 0)} C_1 C_4 (1+\lambda_4^{-1}) \mathbb{E}\left[ \left\| g\left(X_T^\pi\right)-Y_T^\pi\right\|^2\right].
\end{aligned}
\end{equation}
In order to simplify the expressions, we define
\begin{align}\label{def:A1,A2,A3,A4,A5,A6,A7}
    A_1(h) &\coloneqq e^{K_3 T} (1+\lambda_4^{-1}),\quad A_2(h) \coloneqq L^g_x (1+\lambda_4) e^{(K_1+K_3) T} + K_4 h \frac{e^{(K_1+K_3)T}-1}{e^{(K_1+K_3)h}-1},\\
    A_3(h) &\coloneqq K_2 h e^{-K_1 h} \frac{e^{-(K_1+K_3) T}-1}{e^{-(K_1+K_3) h}-1},\quad A_4(h) \coloneqq e^{\max(-K_1 T, 0)} C_1 C_4 C_2 h e^{-K_3h} \frac{e^{-K_3 T}-1}{e^{-K_3 h}-1},\\
    A_5(h) &\coloneqq e^{\max(-K_1 T, 0)} C_1 C_4 C_3  L^f_x  h \frac{e^{K_1 T}-1}{e^{K_1 h}-1},\quad A_6(h) \coloneqq e^{\max(-K_1 T, 0)} C_1 C_4 L^g_x(1+\lambda_4)  e^{K_1 T},\\ 
    A_7(h) &\coloneqq e^{\max(-K_1 T, 0)} C_1 C_4 (1+\lambda_4^{-1}).
\end{align}
Consequently, \eqref{estimate_cal_Y} and \eqref{estimate_cal_X} read as follows
\begin{align}
\mathcal{Y} & \leq A_1(h) \mathbb{E}\left[ \left\| g\left(X_T^\pi\right)-Y_T^\pi\right\|^2\right] + A_{2}(h) \mathcal{X} ,  \label{estimate_cal_Y_2}   \\
\mathcal{X} & \leq A_3(h) \mathcal{Y} + A_4(h) \mathcal{Y}  + A_5(h)\mathcal{X} + A_6(h)\mathcal{X} + A_7(h) \mathbb{E}\left[ \left\| g\left(X_T^\pi\right)-Y_T^\pi\right\|^2\right] .  \label{estimate_cal_X_2}
\end{align}
Next, we solve \eqref{estimate_cal_Y_2}-\eqref{estimate_cal_X_2} such that $\mathcal{Y}$ and $\mathcal{X}$ are both controlled by $\mathbb{E}[\|g(X_T^\pi)-Y_T^\pi \|^2]$. Let 
\begin{equation}\label{def:A_h,B_h}
    \begin{aligned}
        B(h) \coloneqq A_5(h) + A_6(h), \quad A(h) \coloneqq  A_2(h) \left( 1-A_5(h)-A_6(h) \right)^{-1} (A_3(h) + A_4(h)).
    \end{aligned}
\end{equation}
Whenever $B(h)<1$, rearranging the terms in \eqref{estimate_cal_X_2} yields
\begin{equation}\label{ineq: cal_X_with_B}
\begin{aligned}
\mathcal{X} & \leq \left((1 - A_5(h) - A_6(h) \right)^{-1} \left(  \left(A_3(h) + A_4(h)\right) \mathcal{Y} + A_7(h) \mathbb{E}\left[ \left\| g\left(X_T^\pi\right)-Y_T^\pi\right\|^2\right]  \right).
\end{aligned}
\end{equation}
Additionally, if also $A(h)<1$, we can derive the following  by substituting \eqref{ineq: cal_X_with_B} into \eqref{estimate_cal_Y_2}
\begin{equation}\label{ineq: cal_Y_with_AB}
\begin{aligned}
\mathcal{Y}
& \leq \left(1-A(h)\right)^{-1}  \left( A_1(h) + A_2(h) \left(1-B(h) \right)^{-1} A_7(h) \right)
\mathbb{E}\left[ \left\| g\left(X_T^\pi\right)-Y_T^\pi\right\|^2\right].
\end{aligned}
\end{equation}
From \eqref{def:K1,K2,K3,K4,C1} and \eqref{def:C2,C3,C4}, we directly collect the limits
\begin{align}\label{def:bar:K1,K2,K3,K4,C1}
    & \bar{K}_1 = 2 k^b +\lambda_1+L^\sigma_x, \quad\bar{K}_2 = \lambda_1^{-1} L^b_y + L^\sigma_y, \quad\bar{K}_3= 2 k^f +\lambda_2, \quad\bar{K}_4=  L^f_x \lambda_2^{-1}, \\
    & \bar{C}_1 = \lambda_1^{-1} L^b_z,\quad \bar{C}_2 = 2 ( \lambda_3^{-1} L^f_y + \lambda_3 ),\quad \bar{C}_3 = 2 \lambda_3^{-1},\quad \bar{C}_4 = \frac{m}{1-2m L^f_z \lambda_3^{-1}}.
\end{align}
Consequently, from \eqref{def:A1,A2,A3,A4,A5,A6,A7} we directly have
\begin{align}\label{def:bar:A1,A2,A3,A4,A5,A6,A7}
    \bar{A}_1 & = e^{\bar{K}_3 T}  (1+\lambda_4^{-1} ),\quad \bar{A}_2= L^g_x (1+\lambda_4) e^{\left(\bar{K}_1+\bar{K}_3\right) T} + \frac{\bar{K}_4}{\bar{K}_1+\bar{K}_3}  \left(e^{\left(\bar{K}_1+\bar{K}_3\right) T}-1\right),\\
    \bar{A}_3 &= \bar{K}_2 \frac{1-e^{-\left(\bar{K}_1+\bar{K}_3\right) T}}{\bar{K}_1+\bar{K}_3},\quad \bar{A}_4 = e^{\max(-\bar{K}_1 T, 0)} \bar{C}_1 \bar{C}_4 \bar{C}_2 \frac{1-e^{-\bar{K}_3 T}}{\bar{K}_3},\\
    \bar{A}_5 & = e^{\max(-\bar{K}_1 T, 0)} \bar{C}_1 \bar{C}_4 \bar{C}_3  L^f_x  \frac{e^{\bar{K}_1 T}-1}{\bar{K}_1},\quad \bar{A}_6 = e^{\max(-\bar{K}_1 T, 0)} \bar{C}_1 \bar{C}_4 L^g_x (1+\lambda_4 )  e^{\bar{K}_1 T},\\
    \bar{A}_7 &= e^{\max(-\bar{K}_1 T, 0)} \bar{C}_1 \bar{C}_4 (1+\lambda_4^{-1} ),
\end{align}
with the convention $\bar{A}_j=\lim_{h\to 0} A_j(h)$, $j=1,\dots, 7$. If $\bar{K}_1<0$ the expressions above hold only for sufficiently small $h$ such that $K_1<0$. Using the definitions in \eqref{def:A_h,B_h}, it is straightforward to check that $\lim_{h \rightarrow 0} B(h)\eqqcolon \bar{B}$ and $\lim_{h \rightarrow 0} A(h) \eqqcolon \bar{A}$, given by \eqref{def:B_bar} and \eqref{def:A_bar}, respectively.
From \eqref{def:cal_XY}, we get
\begin{align}\label{eq:thm:y:exponential_time_cal}
    \max_{0\leq n\leq N} \mathbb{E} \left[ \left\|\delta X_n\right\|^2\right] &\leq e^{\max(K_1 T, 0)} \mathcal{X},\quad
    \max_{0\leq n\leq N} \mathbb{E} \left[ \left\|\delta Y_n\right\|^2\right]  \leq e^{\max(-K_3 T, 0)} \mathcal{Y},
\end{align}
with $K_1, K_3$ defined in \eqref{def:K1,K2,K3,K4,C1}, both depending on $h$. 
When $\bar{B}<1$ and $\bar{A}<1$, we have that for any sufficiently small $h$ \eqref{ineq: cal_Y_with_AB} holds true. Hence, combining \eqref{ineq: cal_Y_with_AB} with \eqref{eq:thm:y:exponential_time_cal}, we derive that for any sufficiently small $h$
\begin{align}\label{estimate_P}
    \max_{0\leq n\leq N} \mathbb{E} \left[ \left\|\delta Y_n\right\|^2\right]  \leq C\left(\lambda_1, \lambda_2, \lambda_3, \lambda_4\right) \mathbb{E}\left[ \left\| g\left(X_T^\pi\right)-Y_T^\pi\right\|^2\right],
\end{align}
with a constant independent of $h$, depending only on the limits defined in \eqref{def:bar:K1,K2,K3,K4,C1}, \eqref{def:B_bar}, \eqref{def:A_bar}, \eqref{def:bar:A1,A2,A3,A4,A5,A6,A7}, and thus implicitly on $\lambda_1, \lambda_2, \lambda_3, \lambda_4$. Similarly, when $\bar{B}, \bar{A}<1$, combining \eqref{ineq: cal_X_with_B} with \eqref{ineq: cal_Y_with_AB}, we also have that for any sufficiently small $h$
\begin{align}\label{estimate_X}
    \max_{0\leq n\leq N} \mathbb{E} \left[ \left\|\delta X_n\right\|^2\right] \leq C\left(\lambda_1, \lambda_2, \lambda_3, \lambda_4\right) \mathbb{E} \left[\left\| g\left(X_T^\pi\right)-Y_T^\pi\right\|^2\right],
\end{align}
for a constant determined by \eqref{def:bar:A1,A2,A3,A4,A5,A6,A7} which is independent of $h$.

In order to estimate $\mathbb{E}[\|\delta Z_n \|^2]$, we consider \eqref{deltaP_2} from the proof of lemma \ref{lem:estimate1}. Notice that $1-L^f_z/\lambda_2 > 0$ since we require $\lambda_2>L^f_z\geq 0$, then by rearranging the terms in \eqref{deltaP_2} we obtain
\begin{equation} 
\begin{aligned}
 \mathbb{E} \left[ \left\|\delta Z_i\right\|^2\right] h
\leq & \left(1 - L^f_z\lambda_2^{-1} \right)^{-1} \left(  L^f_x \lambda_2^{-1} \mathbb{E}\left[ \left\|\delta X_i\right\|^2 \right] h +  \mathbb{E} \left[\left\|\delta Y_{i+1}\right\|^2 \right] - \mathbb{E}\left[\left\|\delta Y_{i}\right\|^2 \right]  \right. \\
& \left. + \left(2 k^f + \lambda_2 \right) h 
 \mathbb{E} \left[\left\|\delta Y_{i}\right\|^2 \right] \right).
\end{aligned}
\end{equation}
Summing from 0 to $N-1$ and taking the maximum on the right hand side, we gather
\begin{equation}\label{estimate_Q}
\begin{aligned}
\sum_{i=0}^{N-1} \mathbb{E} \left[ \left\|\delta Z_i\right\|^2\right] h
& \leq \begin{aligned}[t]
    \left(1 - L^f_z\lambda_2^{-1} \right)^{-1} \Big( &L^f_x \lambda_2^{-1} T \max_{0\leq n\leq N} \mathbb{E}\left[\left\|\delta X_n\right\|^2\right]\\ 
    &+  \left( \max\{(2k^f + \lambda_2) T, 0\} + 1 \right) \max_{0\leq n\leq N} \mathbb{E}\left[\left\|\delta Y_n\right\|^2\right]  \Big)
\end{aligned}\\
& \leq C(\lambda_1, \lambda_2, \lambda_3,\lambda_4) \mathbb{E}\left[ \left\| g\left(X_T^\pi\right)-Y_T^\pi\right\|^2 \right],  
\end{aligned}
\end{equation}
using \eqref{estimate_X}, \eqref{estimate_P}.

Finally, combining estimates \eqref{estimate_X}, \eqref{estimate_P} and \eqref{estimate_Q} with the convergence of the discrete time approximations such as in theorem \ref{convergence_im}, we prove our statement.
\end{proof}

\section{Interpretation of the conditions in theorem \ref{thm:estimate2}}\label{section:interpretation}
In this section, we apply theorem \ref{thm:estimate2} to special cases of FBSDEs and discuss how the conditions change depending on the coefficients in \eqref{eq:FBSDE}. Furthermore, we illustrate the role of the abstract conditions imposed by \eqref{thm: condition}, and discuss several important heuristic settings under which they are satisfied.

\begin{enumerate}[label=(\arabic*), wide, labelindent=0pt, listparindent=0pt]
\item \textit{Decoupled FBSDE}

In this case, $L^b_y = L^b_z = L^\sigma_y \equiv 0$, which immediately implies that $\bar{B} = \bar{A} \equiv 0$, since both $\bar{C}_1=0$ and $\bar{K}_2=0$. Estimates \eqref{ineq: cal_X_with_B}, \eqref{ineq: cal_Y_with_AB} then reduce to
\begin{equation}
\mathcal{X}  = 0, \qquad \mathcal{Y} \leq e^{K_3 T}  (1+\lambda_4^{-1} ) \mathbb{E} \left[ \left\| g\left(X_T^\pi\right)-Y_T^\pi\right\|^2 \right].
\end{equation}
Consequently, the total errors in the SDE reduce to those of the Euler-Maruyama discretization from \eqref{implicit_scheme}, whereas for the BSDE part a posteriori error term remains in \eqref{eq:main_theorem:estimate}.

\item \label{interpretation:no_z} \textit{Coupled FBSDE with only $Y$ coupling}

In this case $L^b_z=0$ and therefore $\bar{C}_1=0$. We remark that we fully recover the result of \cite{hanlong2020} in this setting, in particular
\begin{align}\label{eq:a_bar:b=0}
    \bar{B}\equiv 0,\quad \bar{A}=\begin{aligned}[t]
        &\Big( L^g_x (1+\lambda_4) e^{\left(\bar{K}_1+\bar{K}_3\right) T} + \frac{\bar{K}_4}{\bar{K}_1+\bar{K}_3}  \big(e^{\left(\bar{K}_1+\bar{K}_3\right) T}-1\big)\Big)\Big( \bar{K}_2 \frac{1-e^{-\left(\bar{K}_1+\bar{K}_3\right) T}}{\bar{K}_1+\bar{K}_3}  \Big),
    \end{aligned}
\end{align}
where the condition $\bar{B}<1$ becomes redundant and is automatically satisfied, whereas $\bar{A}$ has the same expression as the one derived in \cite{hanlong2020}. Moreover, we recover the weak and monotonicity conditions as in \cite[remark 6]{hanlong2020}, which guarantee $\bar{A}<1$. 

\item\label{interpretation:general_case} \textit{Coupled FBSDE in general as in \eqref{eq:FBSDE}}

This is the general setting we considered throughout this paper corresponding to \eqref{eq:FBSDE}. In order to guarantee that the conditions of theorem \ref{thm:estimate2}, and in particular \eqref{thm: condition} are satisfied we need certain requirements about $T$, the constants in $\mathcal{L}$, and choose $\lambda_1$, $\lambda_2$, $\lambda_3$ and $\lambda_4$ in an appropriate way.
Recall that $\bar{B} \equiv \bar{B}(\lambda_1, \lambda_2, \lambda_3, \lambda_4)$ and $\bar{A} \equiv \bar{A}(\lambda_1, \lambda_2, \lambda_3, \lambda_4)$ are functions of all $\lambda$s, defined by \eqref{def:B_bar} and \eqref{def:A_bar}, respectively. We divide the discussion into the following five cases which all have important physical interpretations.

\begin{enumerate}
    \item \textit{Small time duration.} Suppose all other constants, $\lambda_2$, $\lambda_3$ and $\lambda_4$ are fixed. If $T>0$ is sufficiently small, we can choose, for instance, $\lambda_1 = 1/\sqrt{T}$ which implies that $\bar{B}$ is sufficiently close to zero due to the factors $\bar{C}_1$ and $e^{\bar{K}_1 T}-1$. Similarly $\bar{A}$ is sufficiently close to zero as well, due to the scaling factors $\bar{K}_2$, $1 - e^{- (\bar{K}_1 + \bar{K}_3) T}$, $\bar{C}_1$, $1-e^{\bar{K}_3 T}$ in the last term of \eqref{def:A_bar}. Therefore \eqref{thm: condition} is satisfied for sufficiently small time durations $T$.
    
    \item \textit{Weak coupling from BSDE to SDE.} Suppose all other constants, $\lambda_1, \lambda_2$, $\lambda_3$ and $\lambda_4$ are fixed. If $L^b_y>0$, $L^b_z>0$ and $L^\sigma_y>0$ are sufficiently small, then so are the factors $\bar{C}_1$ and $\bar{K}_2$. Notice that $\bar{B}$ is scaled by $\bar{C}_1$, and for $\bar{A}$, the last term in \eqref{def:A_bar} is scaled by both $\bar{C}_1$ and $\bar{K}_2$, and therefore both $\bar{A}$ and $\bar{B}$ are sufficiently close to zero and \eqref{thm: condition} holds.

    \item \label{interpretation:fully_coupled:small_L_b_z} \textit{Weak coupling from SDE to BSDE}. Suppose all other constants, $\lambda_1, \lambda_2$, $\lambda_3$ and $\lambda_4$ are fixed. If $L^g_x>0$ and $L^f_x>0$ are sufficiently small, and additionally $L^b_z>0$ is sufficiently small as well, then both $\bar{B}$ and $\bar{A}$ could be sufficiently close to zero, due to the scaling factors  $L^f_x$, $\bar{C}_1$ and $L^g_x$ in $\bar{B}$, and $L^g_x$ and $\bar{K}_4$ in the first term of \eqref{def:A_bar} for $\bar{A}$. Consequently, \eqref{thm: condition} is satisfied.

    \item\label{interpretation:k_b<0} \textit{Monotonicity in $b$.} Suppose all other constants, $\lambda_3$ and $\lambda_4$ are fixed,  $k^b<0$ is sufficiently negative and $L^b_z>0$ is sufficiently small. We set $\lambda_1 = - 2k^b-\epsilon >0 $ which implies $\bar{K}_1 = -\epsilon + L^\sigma_x <0 $ is fixed for any chosen $\epsilon>L^\sigma_x$. Then, $\bar{B}$ could be sufficiently close to 0 since $\bar{K}_1<0$ is fixed and $\bar{C}_1 = \lambda_1^{-1} L^b_z$ could be sufficiently small. For $\bar{A}$, we directly compute from \eqref{def:A_bar}
\begin{equation}\label{interpretation_barA}
\begin{aligned}
\bar{A} = 
& \left(1 - \bar{B} \right)^{-1} 
 \left(  e^{\max( - \bar{K}_1 T, 0)} \bar{C}_1 \bar{C}_4 \bar{C}_2 \frac{1-e^{-\bar{K}_3 T}}{\bar{K}_3} \left( L^g_x (1+\lambda_4) e^{\left(\bar{K}_1+\bar{K}_3\right) T}  \right)  \right. \\
& \left. + e^{\max(- \bar{K}_1 T, 0)} \bar{C}_1 \bar{C}_4 \bar{C}_2 \frac{1-e^{-\bar{K}_3 T}}{\bar{K}_3}   \left( \frac{\bar{K}_4}{\bar{K}_1+\bar{K}_3} \left(e^{\left(\bar{K}_1+\bar{K}_3\right) T}-1\right) \right)  \right. \\
&  + \left. L^g_x (1+\lambda_4) \bar{K}_2 \frac{  e^{\left(\bar{K}_1+\bar{K}_3\right) T} -1 }{\bar{K}_1+\bar{K}_3} 
+ \bar{K}_2 \bar{K}_4 \frac{ e^{\left( \bar{K}_1+\bar{K}_3\right) T} + e^{-\left( \bar{K}_1+\bar{K}_3\right) T} -2 }{ (\bar{K}_1+\bar{K}_3 )^2}  \right). 
\end{aligned}
\end{equation}
Let us first consider the last two terms in \eqref{interpretation_barA}. Notice that we could fix some sufficiently large $\epsilon>L^\sigma_x$ and $\lambda_2> L^f_z$ such that $\bar{K}_1+\bar{K}_3 =  -\epsilon + L^\sigma_x + 2k^f + \lambda_2 <0$ is fixed and negative enough, and the penultimate term is sufficiently small because the fraction term is decreasing in $\bar{K}_1+\bar{K}_3$ and $\lambda_1$ is sufficiently large. The last term could be sufficiently small as well, since $\bar{K}_1+\bar{K}_3$ is fixed and $\bar{K}_4 = L^f_x / \lambda_2$ can be sufficiently small by choosing a large enough $\lambda_2$. For the remaining first two terms in \eqref{interpretation_barA}, notice that $\bar{K}_1+\bar{K}_3$ is fixed, and for the previously chosen $\lambda_2$, $\bar{K}_3$ is fixed as well, then both two terms are scaled by $\bar{C}_1$ with a sufficiently small $L^b_z$. Combining all these arguments we have $\bar{A}<1$ and conclude that \eqref{thm: condition} is verified.

\item \label{interpretation:k_f<0}\textit{Monotonicity in $f$.} Suppose all other constants, $\lambda_1$, $\lambda_3$ and $\lambda_4$ are fixed. If $k^f<0$ is sufficiently negative, and $L^b_z>0$ is sufficiently small, then it is easy to see that $\bar{B}$ could be sufficiently close to 0 due to the scaling factor $\bar{C}_1 = \lambda_1^{-1} L^b_z$  and the fact that $\bar{B}$ does not depend on $k^f$ and $\lambda_2$. To deal with $\bar{A}$, we set $\lambda_2 = -2k^f - \epsilon$, where $\epsilon>0$ is chosen such that $\bar{K}_1+\bar{K}_3 = 2k^b + \lambda_1 + L^\sigma_x -\epsilon<0$ is negative enough and the penultimate term in \eqref{interpretation_barA} is sufficiently small. Since now $\bar{K}_1+\bar{K}_3$ and $\bar{K}_3 = -\epsilon $ are fixed, the remaining three terms in $\bar{A}$ are all scaled by $\bar{C}_1$ and therefore we have $\bar{A}<1$ guaranteeing \eqref{thm: condition}.
\end{enumerate}

\item \textit{Coupled FBSDE with $b$ only depending on $Z$}

In this case $L^b_y = 0$. However, all constants $\bar{K}_j$, $\bar{C}_j$ for $j=1,2,3,4$ defined in \eqref{def:K1,K2,K3,K4,C1}, \eqref{def:C2,C3,C4} are not zero in general, and therefore the conditions fall back under the general case discussed above. Even in a more special setting with  $L^b_y = L^\sigma_y = L^f_y = k^f= 0$, i.e. when there is no $Y$ coupling in neither the forward nor backward equation, the conditions do not seem to be easier to satisfy. Specifically, we have $\bar{K}_2=0$ in this special setting and consequently there will be one term less in $\bar{A}$ given by \eqref{def:A_bar}, but the expression \eqref{def:B_bar} for $\bar{B}$ remains unchanged as it does not depend on $\bar{K}_2$. On one hand, this reduces some efforts due to the missing term in $\bar{A}$, for example, we do not take care of the last terms discussed in \eqref{interpretation:k_b<0} as they vanish. On the other hand, as $k^f=0$ in this special case, we lost one possible way to make the conditions hold, i.e. \eqref{interpretation:k_f<0} does not apply anymore. In overall, we conclude that coupling in $Z$, even in the special cases mentioned above, induces the need to treat the conditions in theorem \ref{thm:estimate2} under the general framework established by our convergence result.
\end{enumerate}

\begin{remark} The above five cases in \ref{interpretation:general_case} may be viewed as a generalization of the weak and monotonicity conditions stated in \cite{hanlong2020, bender2008time}. One should note that, because of the extra $Z$ coupling, we have to pay an extra price in establishing these five cases, e.g., we need to choose $\lambda$s appropriately instead of fixing them as constants as in \cite{hanlong2020}, and in particular, we require $L^b_z>0$ to be sufficiently small for \ref{interpretation:fully_coupled:small_L_b_z}, \ref{interpretation:k_b<0} and \ref{interpretation:k_f<0}. 
\end{remark}

\begin{remark}\label{remark:decoupling}
    In \cite[remark 2]{hanlong2020}, it is claimed that it is general to consider drifts which only depend on $X$ and $Y$ but not on $Z$, since, by an appropriate change of probability measure, one can always reformulate \eqref{eq:FBSDE} into an equivalent FBSDE whose drift is independent of $Z$ yet its solution coincides with the same quasi-linear PDE. However, we find that there are several important issues with this statement both theoretically and numerically:
    \begin{enumerate}[label=(\arabic*)]
        \item\label{remark:decoupling:practical_point}such approach would change the probability measure under which $X$ is simulated, consequently the training of neural networks would be carried out in a different spatial region, and therefore the algorithm may have poor accuracy around the area of the domain of interest;
        \item\label{remark:decoupling:theoretical_point} it is common that the reformulated FBSDE does not satisfy the theoretical assumptions needed for convergence while the original one does. For instance, a linear $z$ term in $b$ would result in a quadratic term in the reformulated driver, which violates the assumptions of Lipschitz continuity. We illustrate this through Example \ref{sec:ex1} in section \ref{section:numerical} below. Therefore, we believe our theory is a necessary generalization to \cite{hanlong2020} and it is applicable to a wider class of FBSDEs;
        \item with the same approach, one could remove the entire drift to the driver, and simulate a reformulated FBSDE with zero drift, but this, for the two reasons above, is rarely done in practice.
    \end{enumerate}
\end{remark}

Finally, let us derive a lower bound for $\bar{B}$ in \eqref{def:B_bar} by computing the infimum of $\bar{B}$ over all possible choices of $\lambda$s. Notice that $\bar{B}$ does not depend on $\lambda_2$, decreases in $\lambda_3$ and increases in $\lambda_4$, and therefore we shall mainly look at $\lambda_1$. Let 
\begin{align}\label{def:B_l}
    \bar{B}_{\ell} \coloneqq \inf_{\lambda_1>0, \lambda_3> 2mL^f_z, \lambda_4>0} \bar{B}.
\end{align}
\begin{enumerate}[label=(\arabic*)]
\item If $\lambda_1\geq \max(-2k^b-L^\sigma_x, 0)$, then $\bar{B}$ admits a unique stationary point along the $\lambda_1$ direction, and
\begin{equation}
\begin{aligned}
& \bar{B}_{\ell}  = m L^b_z L_x^g e^{( 2k^b + T^{-1}  + L^\sigma_x)T}  T, 
\qquad \arg \inf_{\lambda_1, \lambda_3, \lambda_4} \bar{B} = ( 1/T, +\infty, 0).
\end{aligned}
\end{equation}

\item If $0< \lambda_1< \max(-2k^b-L^\sigma_x, 0)$, then $\bar{B}$ is convex in $\lambda_1$ but there is not stationary point in this range, and
\begin{equation}
\begin{aligned}
& \bar{B}_{\ell} = m \frac{L^b_z}{-2k^b-L^\sigma_x}  L_x^g , 
\qquad 
\arg \inf_{\lambda_1, \lambda_3, \lambda_4} \bar{B} = (-2k^b-L^\sigma_x, +\infty, 0).
\end{aligned}
\end{equation}
\end{enumerate}
This result is particularly useful in practice and can serve as a preliminary test for the convergence of a given FBSDE. In fact, given an equation and all its relevant Lipschitz constants, if we find $\bar{B}_{\ell}\geq 1$, then we know that the conditions of theorem \ref{thm:estimate2} cannot be satisfied and that the deep BSDE algorithm is less likely to converge. On the other hand, if $\bar{B}_{\ell}$ is less than 1 or in particular even close to 0, we may check if $\bar{A}<1$, which can be solved efficiently by a numerical constrained minimization method ranging different $\lambda_1$, $\lambda_2$, $\lambda_3$ and $\lambda_4$. 

\section{Numerical experiments}\label{section:numerical}
We implemented the deep BSDE method in TensorFlow 2.9. 
The errors correspond to the discretized version of the left hand side of \eqref{eq:main_theorem:estimate} and are defined as follows
\begin{align}
    \text{error}(X)&\coloneqq \max_{0\leq n\leq N} \mathbb{E} \left[\left\|\hat{X}_n^\pi - X_{t_{n}}\right\|^2\right],\quad \text{error}(Y)\coloneqq \max_{0\leq n \leq N} \mathbb{E}\left[\left\| \hat{Y}_n^\pi - Y_{t_{n}}\right\|^2 \right],\\
    \text{error}(Z)&\coloneqq 1/N\sum_{n=0}^{N-1}\mathbb{E}\left[\left\| \hat{Z}_n^\pi - Z_{t_{n}}\right\|^2\right],\quad \text{total} = \text{error}(X) + \text{error}(Y) + \text{error}(Z).
\end{align}
We also define relative $L^2$ approximation errors by $\text{error}(X)/\mathbb{E}[\|X_{t_{n}}\|^2]$ and similarly for $Y, Z$, total.
In what follows, the true expectations are approximated over a Monte Carlo sample of size $2^{12}$. Given a classical solution to the corresponding quasi-linear PDE used for decoupling, we gather a reference solution to the associated FBSDE \eqref{eq:FBSDE} by an Euler-Maruyama simulation with $N'=10^4$ time steps, in order to guarantee that the time discretization error of the reference solution is negligible compared to the approximation errors incurred via the deep BSDE method. In all experiments below, we use neural networks with hyperbolic tangent activation, an input layer of width $d$, $2$ hidden layers $30+d$ neurons wide each, and an output layer of appropriate dimensions depending on the process approximated. As an optimization strategy, we use the Adam optimizer with default initializations and a learning rate schedule of exponential decay, starting from $10^{-2}$ with a decay rate of $10^{-2}$.  For a fixed $N$ we perform $2^{14}$ SGD iterations, and for each iteration we take an independent sample of $2^{10}$ trajectories of the underlying Brownian motion. All experiments below were run on an Dell Alienware Aurora R10 machine, equipped with an AMD Ryzen 9 3950X CPU (16 cores, 64Mb cache, 4.7GHz) and an Nvidia GeForce RTX 3090 GPU (24Gb). In order to assess the inherent stochasticity of both the regression Monte Carlo method and the SGD iterations, we run each experiment 5 times and report on the mean and standard deviations of the resulting independent approximations. All operations were carried out with single precision.
\subsection{Example 1}\label{sec:ex1}
\begin{figure}
    \centering
    \begin{subfigure}[t]{\figsize\textwidth}
        \centering
        \includegraphics[width=\textwidth]{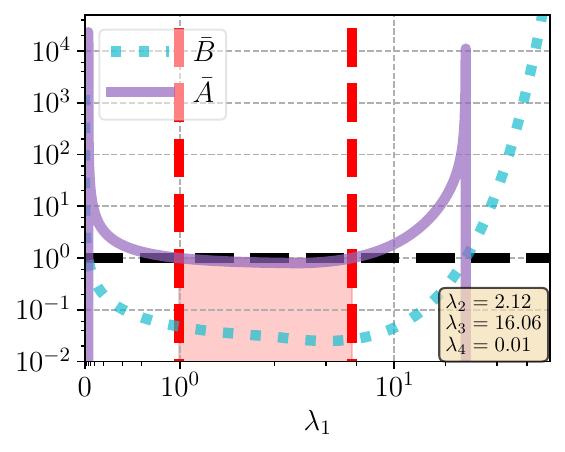}
        \caption{$\bar{B}, \bar{A}$ as functions of $\lambda_1$ for given $(\lambda_2, \lambda_3, \lambda_4)$ in case of \eqref{eq:ex1}. Dotted vertical red lines mark the endpoints of the interval where $\bar{B}, \bar{A}<1$, and the shaded red area the subset of the plane where the sufficient conditions of theorem \ref{thm:estimate2} are satisfied.}
        \label{fig:ex1:ab}
    \end{subfigure}\hspace{0.5em}
    \begin{subfigure}[t]{\figsize\textwidth}
        \centering
        \includegraphics[width=\textwidth]{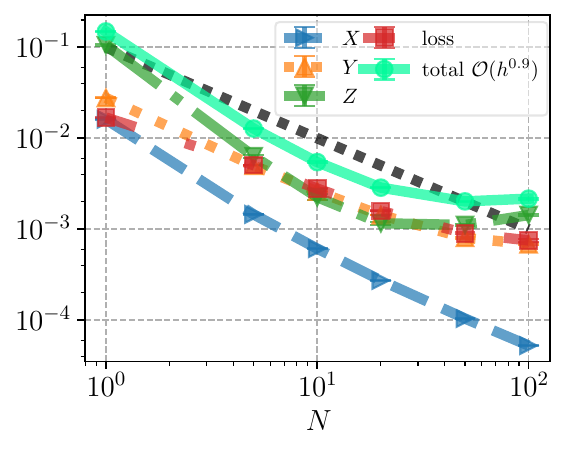}
        \caption{Convergence in $N$. Empirical convergence rate in labels. Dotted black line indicates the expected $\mathcal{O}(h)$ rate predicted by theorem \ref{thm:estimate2}.}
        \label{fig:ex1:conv}
    \end{subfigure}
    \caption{Example \ref{sec:ex1}. $T=0.25, X_0=(\pi/4,\dots,\pi/4)$.}
    \label{fig:ex1}
\end{figure}
The following example is a modified version of the one in \cite{bender2008time}, where in order to demonstrate our theoretical extension we include $Z$ coupling in the drift of the forward diffusion. The coefficients of the FBSDE system \eqref{eq:FBSDE} read as follows
\begin{align}\label{eq:ex1}
    b(t, x, y, z) &= \kappa_y \bar{\sigma} y \mathbf{1}_d + \kappa_z z^\top,\quad \sigma(t, x, y) = \bar{\sigma} yI_d,\quad g(x) = \sum_{i=1}^d \sin(x_i),\\
    f(t, x, y, z) &= \begin{aligned}[t]
        &-ry + 1/2 e^{-3r(T - t)} \bar{\sigma}^2 (\sum_{i=1}^d \sin(x_i))^3- \kappa_y \sum_{i=1}^d z_i - \kappa_z \bar{\sigma} e^{-3r (T-t)} \sum_{i=1}^d \sin(x_i)\sum_{i=1}^d \cos^2 (x_i),
    \end{aligned}
\end{align}
with $q=1, d=m$.
The analytical solution pair to the backward equation is given by
\begin{align}\label{eq:ex1:analytical_solution}
    y(t, x) = e^{-r(T-t)}\sum_{i=1}^d \sin(x_i),\quad z_i(t, x) = e^{-2r(T-t)} \bar{\sigma} \big(\sum_{j=1}^d \sin(x_j)\big)\cos(x_i).
\end{align}
We note that the equation above falls under the theoretical assumptions of section \ref{section:convergence}. In particular, we get the following set of values for the corresponding constants
$L_x^g = d, L_y^b=2(\kappa_y \bar{\sigma})^2, L_z^b=2\kappa_z^2, L_y^\sigma = d\bar{\sigma}^2, L_x^f = 3/2d(3\bar{\sigma}^2d^2/2+2\kappa_z \bar{\sigma}d)^2, L_y^f=18r^2, L_z^f=3.6 d\kappa_y^2, k^f=-r, k^b=L_x^\sigma=L_z^\sigma=L_x^b=0$. We consider the equation in $d=10$ dimensions. The strength of coupling is determined by the values $\bar{\sigma}, r, \kappa_y, \kappa_z$. In order to satisfy the sufficient conditions of theorem \ref{thm:estimate2}, we put $r=1, \bar{\sigma}=0.1, \kappa_y=10^{-1}, \kappa_z=10^{-2}$ and $T=0.25, X_0=(\pi/4,\dots,\pi/4)$. Convergence results are collected in figure \ref{fig:ex1}. Figure \ref{fig:ex1:ab} shows that the conditions of theorem \ref{thm:estimate2} are indeed satisfied, there exists a quadruple $(\lambda_1, \lambda_2, \lambda_3, \lambda_4)$, such that $\bar{B}, \bar{A}$ defined by \eqref{def:A_bar}, \eqref{def:B_bar} admit $\max(\bar{B}, \bar{A})<1$. In particular, for the fixed $\lambda_2, \lambda_3, \lambda_4$ we mark the interval of admissible $\lambda_1$s such that the sufficient conditions are satisfied within the shaded red area. Figure \ref{fig:ex1:conv} depicts the convergence of the deep BSDE method. Its most important implications are as follows. The convergence is only guaranteed in a \emph{posteriori} sense. In fact, as can be seen the convergence only shows the expected $\mathcal{O}(h)$ behavior whenever the loss function corresponding to the last term of \eqref{eq:main_theorem:estimate} is dominated by the discretization error. In particular, for $N\in\{50, 100\}$ we see that the approximation errors of $Z$ begin to stall and the total approximation errors are dominated by the loss function. This indicates that for very fine time grids one needs to make sure that the loss is appropriately minimized when trying to recover discretization errors. Given the global minimization structure of the deep BSDE method the corresponding optimization problem becomes more difficult with an increased number of time steps. This demonstrates a clear trade-off between discretization and optimization, which is fully explained by theorem \ref{thm:estimate2} and should be carefully considered in applications.
Nevertheless, we get an empirical convergence rate of $\mathcal{O}(h^{0.9})$ for all time points, and accounting for the reasoning above we recover the predicted rate of our convergence analysis.

Furthermore, let us return to remark \ref{remark:decoupling}. In particular, Han and Long in \cite[Remark 2]{hanlong2020} claim that the setting of $Z$ independent drift is general, since, due to the connections with the associated quasi-linear PDEs, one can always move the $Z$ dependence from the drift to the driver. In order to complement our arguments against this reasoning in remark \ref{remark:decoupling}, we provide a numerical demonstration of the points raised therein. One can derive a \emph{reformulated} FBSDE system which is \emph{decoupled in $Z$} and whose solution will coincide with \eqref{eq:ex1:analytical_solution}. This equation has a modified drift and driver
\begin{align}\label{eq:ex1:decoupled}
    \tilde{b}(t, x, y) = \kappa_y \bar{\sigma}y 1_d,\quad \tilde{f}(t, x, y, z)= f(t, x, y, z) + \kappa_z \|z\|^2/(\bar{\sigma}y),
\end{align}
whereas the rest of the coefficients remain the same as in \eqref{eq:ex1}. First, notice that even though \eqref{eq:ex1} satisfies the theoretical assumptions of theorem \ref{thm:estimate2}, the reformulated FBSDE \eqref{eq:ex1:decoupled} does not. In particular, $\tilde{f}$ is not Lipschitz continuous in $y, z$ which renders the results of theorem \ref{thm:estimate2}, or \cite{hanlong2020} as a limit case, inapplicable. This demonstrates point-\ref{remark:decoupling:theoretical_point} from remark \ref{remark:decoupling}. Nonetheless, as our convergence analysis only gives sufficient conditions one can still run the deep BSDE algorithm and find satisfactory results without theoretical guarantees. Table \ref{tab:ex1:decoupling} shows that for equation \eqref{eq:ex1:decoupled} this is not case. Running the algorithm on the reformulated FBSDE \eqref{eq:ex1:decoupled} results in diverging errors. In fact, due to the singularity arising in the driver $\tilde{f}$, the backward equation blows up as $N$ increases, which results in the forward equation also exploding due to the coupling. On the contrary, the original equation \eqref{eq:ex1} with $Z$ coupling converges as predicted by theorem \ref{thm:estimate2} and also illustrated by figure \ref{fig:ex1}. This observation demonstrates point-\ref{remark:decoupling:practical_point} from remark \ref{remark:decoupling} and implies that the Lipschitz features in our analysis are crucial also in practice, in order to avoid such explosion of the coupled forward diffusion. This is in line with related results in the literature, see \cite[pg.170]{bender2008time}. Overall, we conclude that the framework of $Z$ coupling in the drift cannot in general be circumvented neither theoretically nor numerically, and one needs to rely on our convergence result in theorem \ref{thm:estimate2} instead.
\begin{table}
    \caption{Comparison on the convergence of the deep BSDE algorithm between \eqref{eq:ex1} and \eqref{eq:ex1:decoupled}. Numbers correspond to the mean(std.dev.) of the total approximation errors of $5$ independent runs of the algorithm. $T=0.25, X_0=(\pi/4,\dots,\pi/4).$} 
    \resizebox{\textwidth}{!}{
    \begin{tabular}{l|ccccccc}
        $N$ & $1$ & $5$ & $10$ & $20$ & $50$ & $100$ \\
        \hline
        total-Eq.\eqref{eq:ex1} & $\num{1.49e-1}(\num{3e-3})$ & $\num{1.28e-2}(\num{3e-4})$ & $\num{5.5e-3}(\num{2e-4})$ & $\num{2.86e-3}(\num{5e-5})$ & $\num{2.02e-3}(\num{4e-5})$ & $\num{2.16e-3}(\num{6e-5})$\\
        total-Eq.\eqref{eq:ex1:decoupled} & $\num{1.52e-2}(\num{3e-3})$ & $\num{4e+2}(\num{2e+2})$ & $\num{2e+6}(\num{4e+6})$ & NaN & NaN & NaN
    \end{tabular}
    }
    \label{tab:ex1:decoupling}
\end{table}

\subsection{Example 2}
The following $d=25$ dimensional example is related to a linear-quadratic stochastic control problem appearing in \cite[example 3]{andersson2023}, which is defined by the following set of coefficients
\begingroup
\allowdisplaybreaks
\begin{align}\label{eq:ex2:lq_coefficients}
    M_x&=-\text{diag}(1, 2, 3, 1, 2, 3, 1, 2, 3, 1, 2, 3, 1, 2, 3, 1, 2, 3, 1, 2, 3, 1, 2, 3, 1),\\ 
    M_u&=(1, 1, 0.5, 1, 0, 0, 1, 1, 0.5, 1, 0, 0, 1, 1, 0.5, 1, 0, 0, 1, 1, 0.5, 1, 0, 0, 1)^\top,\\
    M_c&=\begin{aligned}[t]
        -M_x(&-0.2, -0.1, 0, 0, 0.1, 0.2, -0.2, -0.1, 0, 0, 0.1, 0.2,\\
        &-0.2, -0.1, 0, 0, 0.1, 0.2, -0.2, -0.1, 0, 0, 0.1, 0.2, -0.2)^\top,
    \end{aligned}\\
    \Sigma &= \begin{aligned}[t]
        \text{diag}(&0.15, 0.15, 0.25, 0.25, 0.25, 0.25, 0.25, 0.25, 0.25, 0.25, 0.25,\\ &0.25,0.15,
        0.15, 0.25, 0.25, 0.25, 0.25, 0.25, 0.25, 0.25, 0.25, 0.25, 0.25, 0.25),
    \end{aligned}\\
    R_x &= \begin{aligned}[t]
        2\text{diag}(&25, 1, 25, 1, 25, 1, 25, 1, 25, 1, 25, 1, 25, 1, 25, 1, 25, 1, 25, 1, 25, 1, 25, 1, 25),\quad R_u = 2,
    \end{aligned}\\
    G &= \begin{aligned}[t]
        2\text{diag}(&25, 25, 25, 25, 25, 25, 1, 25, 1, 25, 1, 25, 25, 25, 25, 25, 25, 25, 1, 25, 1, 25, 1, 25, 1).
    \end{aligned}
\end{align}
\endgroup
One can derive an associated FBSDE system via either dynamic programming (DP) or the stochastic maximum principle (SMP), see e.g. \cite{yong1999stochastic}. The corresponding equations' coefficients in \eqref{eq:FBSDE} then take the following form in case of DP
\begin{align}\label{eq:ex2:dp}
    b_\text{DP}(t, x, y, z) &= M_x x - M_u R_u^{-1}M_u^\top(z \Sigma^{-1})^\top,\quad  \sigma_\text{DP}(t, x, y) = \Sigma,\\
    f_\text{DP}(t, x, y, z)&= 1/2(x^\top R_x x + z\Sigma^{-1}(R_u^{-1}M_u^\top)^\top M_u^\top (z\Sigma^{-1})^\top),\quad g_\text{DP}(x)= 1/2x^\top G x,
\end{align}
with $q=1, d=m=25$; and in case of SMP
\begin{align}\label{eq:ex2:smp}
   b_\text{SMP}(t, x, y, z) &= M_x x + M_u R_u^{-1} M_u^\top y,\quad \sigma_\text{SMP}(t, x, y) = \Sigma,\\
    f_\text{SMP}(t, x, y, z)&= -R_x x +M_x y,\quad g_\text{SMP}(x)=- G x,
\end{align}
with $q=d=m=25$.
The main difference between the two formulations is that \eqref{eq:ex2:dp} leads to an FBSDE where coupling into $b$ occurs through $Z$, whereas in \eqref{eq:ex2:smp} only through $Y$. Furthermore, the first equation gives a scalar-valued BSDE, whereas \eqref{eq:ex2:smp} is a vector-valued one. Both equations admit semi-analytical solutions given by the numerical resolution of a system of Ricatti ODEs.
For details, we refer to \cite{huang_convergence_2025} and the references therein. 
\begin{remark}\label{remark:lipschitz_lq_problem}
    Notice that the dynamic programming FBSDE \eqref{eq:ex2:dp} does not satisfy the Lipschitz conditions imposed in section \ref{section:convergence}. In fact, $g_\text{DP}, f_\text{DP}$ in \eqref{eq:ex2:dp} are quadratic in $x, z$. Nevertheless, one can use a localization argument, and consider the equation over a compact domain such that the corresponding coefficients become Lipschitz continuous with a constant depending on the width of the domain. We choose truncation radiuses based on upper bounds for $99\%$ quantiles of $\|X_t\|$, $\|Z_t\|$ computed over an independent Monte Carlo simulation consisting of $2^{20}$ paths using the semi-analytical reference solution. This results in a negligible truncation error and truncation radiuses $r_x=1, r_z=10$, with which we obtain a Lipschitz continuous approximation of $g_\text{DP}, f_\text{DP}$ for which the constants in section \ref{section:convergence} can be computed even in the case of \eqref{eq:ex2:dp}, and read as follows $L_x^g = r_x^2\|G\|^2_2/2, L_x^b = 2\|M_x\|^2_2, L_z^b = 2\|M_uR_u^{-1}M_u^\top (\Sigma^{-1})^\top\|_2^2, L_x^f = r_x^2\|R_x\|^2_2, L_z^f=r_z^2\|\Sigma^{-1}(R_u^{-1}M_u^\top)^\top M_u^\top (\Sigma^{-1})^\top\|_2^2, k^b=-1, L_y^b=L^\sigma_x=L^\sigma_z=L_y^f=k^f=0$.
\end{remark}

\subsubsection{Non-convergence of the deep BSDE method}\label{sec:ex2:1}
\begin{figure}
    \centering
    \begin{subfigure}[t]{\figsize\textwidth}
        \centering
        \includegraphics[width=\textwidth]{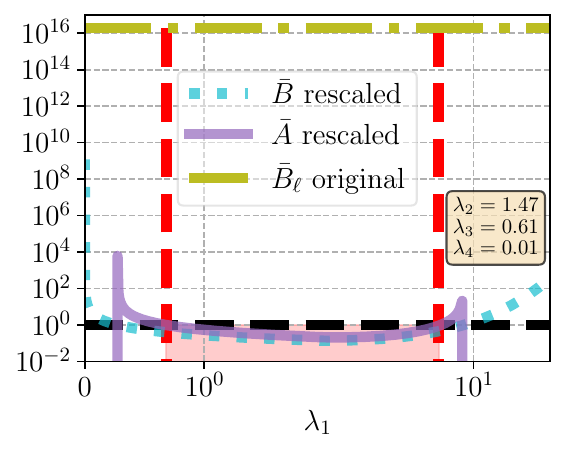}
        \caption{$\bar{B}, \bar{A}$ as functions of $\lambda_1$ for given $(\lambda_2, \lambda_3, \lambda_4)$ in case of the "rescaled" version of \eqref{eq:ex2:dp}. $\bar{B}_\ell$ as lower bound for $\bar{B}$ in case of the "original" version. Dotted vertical red lines mark the endpoints of the interval where $\bar{B}, \bar{A}<1$, and the shaded red area the subset of the plane where the sufficient conditions of theorem \ref{thm:estimate2} are satisfied.}
        \label{fig:ex2:nonconv:ab}
    \end{subfigure}\hspace{0.5em}
    \begin{subfigure}[t]{\figsize\textwidth}
        \centering
        \includegraphics[width=\textwidth]{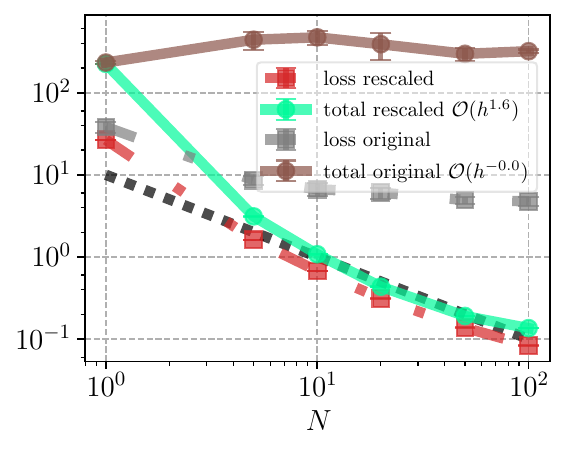}
        \caption{Convergence in $N$. Empirical convergence rates in labels. Dotted black line indicates the expected $\mathcal{O}(h)$ rate predicted by theorem \ref{thm:estimate2}.}
        \label{fig:ex2:nonconv:conv}
    \end{subfigure}
    \caption{Example \ref{sec:ex2:1}. Comparison of the Deep BSDE method on \eqref{eq:ex2:dp} between coefficients as in \eqref{eq:ex2:lq_coefficients} (original) and $M_u$ replaced by $M_u/150$ (rescaled). $T=1/2, X_0=(0.1, \dots, 0.1)$.}
    \label{fig:ex2:nonconv}
\end{figure}
Let us first focus on the FBSDE stemming from the dynamic programming principle. In \cite{andersson2023} it was observed that the deep BSDE method does not converge for the FBSDE defined by \eqref{eq:ex2:dp} with coefficients as in \eqref{eq:ex2:lq_coefficients} and $T=1/2$, $X_0=(0.1, \dots, 0.1)$. Earlier convergence analyses such as \cite{hanlong2020} could not justify this phenomenon, as in \eqref{eq:ex2:dp} the coupling into the drift takes places via $Z$, which fell out of the framework of the aforementioned paper. Our generalization provided by theorem \ref{thm:estimate2} enables the treatment of such equations, and in particular explains the non-convergence phenomenon. The problem lies in the \emph{strength} of the coupling of $Z$ into the forward diffusion. In order to demonstrate this, we consider two versions of \eqref{eq:ex2:dp}, which differ in the coefficient $M_u$. One labelled as \emph{"original"}, where the coefficients of the corresponding LQ problem are as in \eqref{eq:ex2:lq_coefficients}, and a \emph{"rescaled"} version where $M_u$ is replaced by $M_u/150$ and all other coefficients remain the same.
Our findings are illustrated by figure \ref{fig:ex2:nonconv}.
In particular, figure \ref{fig:ex2:nonconv:ab} depicts the contraction constants $\bar{B}, \bar{A}$ defined by \eqref{def:B_bar}, \eqref{def:A_bar} appearing in theorem \ref{thm:estimate2} for both versions. As can be seen, in case of the "original" equation one gets a lower bound $\bar{B}_\ell$ defined by \eqref{def:B_l} which is of $\mathcal{O}(10^{16})$. In fact, this implies that the conditions of theorem \ref{thm:estimate2} can never be satisfied for the original version of \eqref{eq:ex2:dp}. However, as is also suggested by figure \ref{fig:ex2:nonconv:ab}, decreasing the strength of the coupling by the given rescaling of $M_u$ we get an equation whose $\bar{B}, \bar{A}$ satisfy the sufficient conditions \eqref{thm: condition}. Motivated by this, we collect the convergence of the total approximation errors in figure \ref{fig:ex2:nonconv:conv}. In line with the discussion above, we find that the deep BSDE method does not converge in the "original" case, whereas it does converge for the "rescaled" version, once we make sure that there exist appropriate contraction constants $\bar{B}, \bar{A}<1$ such that the sufficient conditions of theorem \ref{thm:estimate2} are satisfied. The empirical rate of convergence is of $\mathcal{O}(h^{1.6})$. Note that this example illustrates the weak coupling condition described in point \ref{interpretation:fully_coupled:small_L_b_z} of section \ref{section:interpretation}.
We emphasize that $\bar{B}$ is inherent to the $Z$ coupling and it is the main novelty of our convergence analysis.

\subsubsection{Stochastic control via DP or SMP}\label{sec:ex2:2}
\begin{figure}
    \centering
    \begin{subfigure}[t]{\figsize\textwidth}
        \centering
        \includegraphics[width=\textwidth]{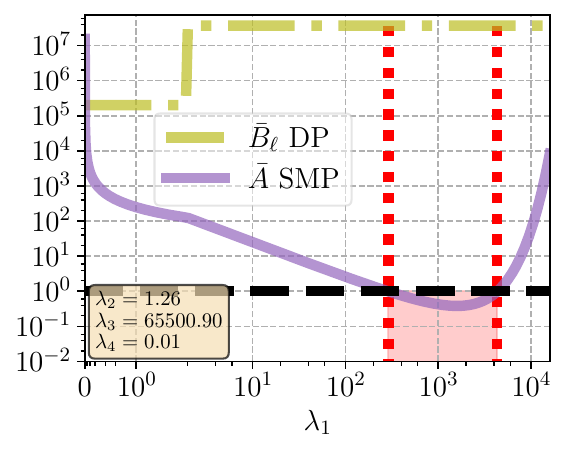}
        \caption{$\bar{A}$ as function of $\lambda_1$ for given $(\lambda_2, \lambda_3, \lambda_4)$ in case of the SMP equation in \eqref{eq:ex2:smp}. $\bar{B}_\ell$ as lower bound for $\bar{B}$ in case of the DP equation in \eqref{eq:ex2:dp}. Dotted vertical red lines mark the endpoints of the interval where $\bar{A}<1$, and the shaded red area the subset of the plane where the sufficient conditions of theorem \ref{thm:estimate2} are satisfied.}
        \label{fig:ex2:DPvsSMP:ab}
    \end{subfigure}\hspace{0.5em}
    \begin{subfigure}[t]{\figsize\textwidth}
        \centering
        \includegraphics[width=\textwidth]{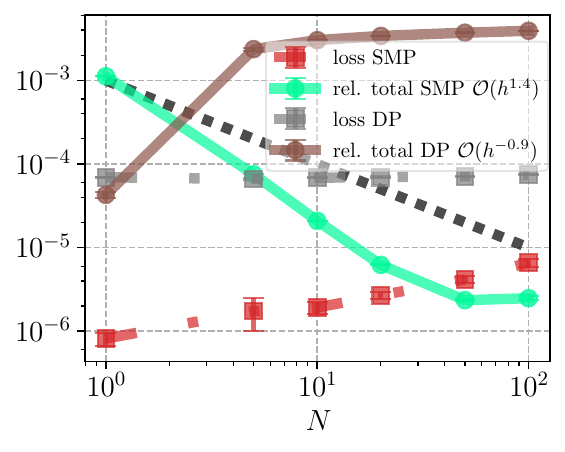}
        \caption{Convergence in $N$. Empirical convergence rates in labels. Dotted black line indicates the expected $\mathcal{O}(h)$ rate predicted by theorem \ref{thm:estimate2}.}
        \label{fig:ex2:DPvsSMP:conv}
    \end{subfigure}
    \caption{Example \ref{sec:ex2:2}. Comparison between the FBSDEs derived via dynamic programming \eqref{eq:ex2:dp} and the stochastic maximum principle \eqref{eq:ex2:smp} approximated by the deep BSDE method. Coefficients as in \eqref{eq:ex2:lq_coefficients}, $T=10^{-3}, X_0=(0.1, \dots, 0.1)$.}
    \label{fig:ex2:DPvsSMP}
\end{figure}
In \cite{huang_convergence_2025} a convergence analysis has been given in the context of solving stochastic control problems with the deep BSDE method applied on the FBSDE system derived through the stochastic maximum principle (SMP) similar to 
\eqref{eq:ex2:smp}. This analysis provided a natural extension to the works of Han and Long by extending \cite{hanlong2020} to vector-valued settings. 
In \cite{andersson2023} it was found that for certain FBSDEs derived in the dynamic programming framework, such as \eqref{eq:ex2:dp}, the Deep BSDE method does not converge. On the other hand, the authors of the present paper found in \cite{huang_convergence_2025} that the same problems tackled by the stochastic maximum principle lead to an FBSDE \eqref{eq:ex2:smp} for which the deep BSDE method does converge to the unique solution triple. Our results in theorem \ref{thm:estimate2} provide a natural explanation for these empirical findings. In fact, the problem lies in the $Z$ coupling in \eqref{eq:ex2:dp}, and in particular the value of the contraction constant $\bar{B}$ defined in \eqref{def:B_bar}. Conversely, \eqref{eq:ex2:smp} derived from the stochastic maximum principle has coupling only in $Y$, not in $Z$, leading to $\bar{B}\equiv 0$. In order to illustrate this, we ran the deep BSDE algorithm for both \eqref{eq:ex2:dp} and \eqref{eq:ex2:smp} with coefficients defined by \eqref{eq:ex2:lq_coefficients}, a short time horizon $T=10^{-3}$ and $X_0=(0.1, \dots, 0.1)$.

First, notice that only \eqref{eq:ex2:smp} satisfies the Lipschitz conditions imposed in section \ref{section:convergence}, with corresponding constants $L^g_x = \|G\|_2^2, L_x^b=2\|M_x\|^2_2, L_y^b=2\|M_uR_u^{-1}M_u^\top\|_2^2, L_x^f=2\|R_x\|_2^2, L_y^f=2\|M_x\|_2^2$, $k^f=k^b=-1, L_z^b=L_x^\sigma=L_y^\sigma=L_z^\sigma=L_z^f=0$. Hence, only the SMP formulation guarantees direct applicability of the convergence results. Nevertheless, with the localization argument in remark \ref{remark:lipschitz_lq_problem} one can find an accurate Lipschitz continuous approximation of the DP equation \eqref{eq:ex2:dp} for which the sufficient conditions can also be checked. More importantly, \eqref{eq:ex2:dp} and \eqref{eq:ex2:smp} also differ in the type of coupling. Namely, in case of the former, $Z$ couples into the forward diffusion, whereas in case of the latter only $Y$ does. This in particular implies that besides the different Lipschitz constants, the two equations also differ in terms of the sufficient conditions of \eqref{eq:main_theorem:estimate}. In fact, for the SMP equation \eqref{eq:ex2:smp}, there is no coupling in $Z$ which implies $\bar{B}\equiv 0$, see also discussion in section \ref{section:interpretation}. Hence \eqref{eq:ex2:smp} reduces to the theoretical framework of \cite{huang_convergence_2025, hanlong2020} under which it is sufficient for $\bar{A}<1$ to hold. On the other hand, \eqref{eq:ex2:dp} couples through $Z$, which in light of theorem \ref{thm:estimate2} implies that $\bar{B}, \bar{A}<1$ need to hold simultaneously, leading to stronger conditions to hold.

The above discussion is illustrated by figure \ref{fig:ex2:DPvsSMP}. From figure \ref{fig:ex2:DPvsSMP:ab}, we find that in case of the dynamic programming equation \eqref{eq:ex2:dp} $\bar{B}$ admits a lower bound $\bar{B}_\ell$ defined by \eqref{def:B_l} which is of $\mathcal{O}(10^5)$. In particular, this implies that the dynamic programming formulation can never satisfy the sufficient conditions imposed by theorem \ref{thm:estimate2}. On the contrary, under the SMP formulation we find a range of $\lambda_1, \lambda_2, \lambda_3, \lambda_4$ such that $\bar{A}<1$ meaning that the convergence criteria are met. Motivated by these conditions, the convergence of the deep BSDE method is collected in figure \ref{fig:ex2:DPvsSMP:conv} for both equations. In line with our previous observations, we find that the method converges with an empirical rate of $\mathcal{O}(h^{1.4})$ for the SMP equation \eqref{eq:ex2:smp}, which in this particular case is even faster than the rate predicted by \cite{hanlong2020, huang_convergence_2025}. On the other hand, for the FBSDE derived via dynamic programming we find that the deep BSDE does not converge. As shown by figure \ref{fig:ex2:DPvsSMP:conv} this is not due to the posteriori nature of the estimate in theorem \ref{thm:estimate2}, as the errors are growing even as the loss functional decreases. The critical phenomenon is indeed the coupling in $Z$, and the extra conditions it imposes as predicted by theorem \ref{thm:estimate2}. These findings complement our earlier convergence results in the context of stochastic control tackled by the deep BSDE method and the stochastic maximum principle \cite{huang_convergence_2025}. In particular, our generalization in theorem \ref{thm:estimate2} explains our empirical findings in \cite{huang_convergence_2025} suggesting that for a large class of stochastic control problems, deriving an associated FBSDE through SMP leads to a system which is more tractable by deep BSDE formulations. The crucial property here is the lack of $Z$ coupling leading to milder conditions to ensure convergence according to our new convergence result in theorem \ref{thm:estimate2}.

\section{Conclusion}
In this paper a generalized proof for the convergence of the deep BSDE method was given. Our main contributions can be summarized as follows. We extended the convergence analysis of \cite{hanlong2020} to FBSDEs with fully-coupled drift coefficients. Such an extension is essential in practice as it enables the treatment of FBSDEs stemming from stochastic optimal control problems.
Our theory provides a unified framework and, in particular, includes earlier results from the literature as limit cases. Due to the extra $Z$ coupling, the final posteriori error estimate stated in theorem \ref{thm:estimate2} requires an additional condition expressed by \eqref{thm: condition}. These sufficient conditions are directly verifiable for any equation, and as shown in section \ref{section:interpretation}, they are satisfied under heuristic settings such as weak coupling, short time duration or monotonicity. Moreover, as demonstrated in section \ref{sec:ex2:1}, our theory explains the non-convergence of the deep BSDE method observed in recent literature, and provides direct guidelines to avoid such issues and ensure convergence in practice.
Several numerical experiments were presented for high-dimensional equations, which support and highlight key features of the theoretical findings. 
\paragraph{Acknowledgement}
The first author acknowledges financial support from the Peter Paul Peterich Foundation via the TU Delft University Fund. The second author would like to thank the China Scholarship Council (CSC)
for the financial support. 

\bibliographystyle{unsrt} 
\bibliography{ref.bib} 


\end{document}